\documentclass[10pt,a4paper]{article}
\usepackage{a4wide}
\usepackage{amsmath,amsthm,amsfonts,amssymb}
\usepackage{graphicx}
\usepackage{microtype}
\usepackage{lmodern}
\usepackage{mathrsfs}
\usepackage[final]{showkeys}
\usepackage{color}

\usepackage{tikz}
\usetikzlibrary{positioning,automata,arrows,decorations,backgrounds}

\usepackage{oldgerm}
\DeclareFontFamily{U}{yswab}{}
\DeclareFontShape{U}{yswab}{m}{n}{<-> yswab}{}

\numberwithin{equation}{section}

\newtheorem{teo}{Theorem}[section]
\newtheorem{lem}{Lemma}[section]
\newtheorem{defn}{Definition}[section]
\newtheorem{obs}{Remark}[section]

\def\bbR{{\mathbb R}}
\def\bbZ{{\mathbb Z}}
\def\bbE{{\mathbb E}}
\def\bbV{{\mathbb V}}
\def\bbP{{\mathbb P}}

\title{Random Graphs Associated to some Discrete and Continuous Time Preferential Attachment Models}

\author{Angelica Pachon, Federico Polito \& Laura Sacerdote\\
	\footnotesize{Mathematics Department ``G.~Peano'', University of Torino, Italy}}

\begin{document}

	\maketitle

	\abstract
	
		We give a common description of Simon, Barab\'asi--Albert,  II-PA and Price growth models,
		by introducing suitable random graph processes with preferential attachment
		mechanisms. Through the II-PA model, we prove the conditions for which the asymptotic degree distribution of
		the Barab\'asi--Albert model coincides with the asymptotic in-degree distribution of the Simon
		model. Furthermore, we show that when the number of vertices in the Simon model
		(with parameter $\alpha$) goes to infinity, a portion of them  behave as a Yule model with parameters $(\lambda,\beta) = (1-\alpha,1)$,
		and through this relation we explain why asymptotic properties of a random vertex in Simon model,
		coincide with the asymptotic properties of a random genus in Yule model. 
		As a by-product of our analysis, we prove the explicit expression of the in-degree distribution
		for the II-PA model, given without proof in \cite{Newman2005}.	
		References to traditional and recent applications of the these models are also discussed.

		\medskip
		

		\medskip		
		

	\section{Introduction}
	
		A large group of networks growth models can be classified as \emph{preferential
		attachment models}.
		In the simplest preferential attachment mechanism an edge connects a newly created node to one of those already present in
		the network with a probability proportional to the number of their edges.
		
		Typically what is analyzed for these models are properties related both to the growth of the number of edges
		for each node and to the growth of the number of nodes.

		After the seminal paper by Barab\'{a}si and Albert \cite{Barabasi1999}, models admitting
		a preferential attachment mechanism have been successfully
		applied to the growth of different real world networks, such as, amongst others, physical, biological or social networks.
		The typical feature revealing a
		preferential attachment growth mechanism is the presence of power-law
		distributions, e.g., for the degree (or in-degree) of a node selected uniformly at random.

		Despite its present success, the preferential attachment paradigm is not new. In fact it
		dates back to a paper by Udny Yule \cite{Yule1925}, published in 1925 and regarding the development
		of a theory of macroevolution. Specifically the study concerned the
		time-continuous process of creation of genera and the evolution of species belonging to them.
		Yule proved that when
		time goes to infinity, the limit distribution of the number of species in a
		genus selected uniformly at random has a specific form and exhibits
		a power-law behavior in its tail. Thirty years later,
		the Nobel laureate Herbert A.\ Simon proposed a time-discrete preferential attachment model to describe the
		appearance of new words in a large piece of a text. Interestingly enough, the limit
		distribution of the number of occurrences of each word, when the number of
		words diverges, coincides with that of the number of species belonging to the randomly
		chosen genus in the Yule model, for a specific choice of the parameters. This fact
		explains the designation \emph{Yule--Simon distribution} that is commonly assigned to that limit
		distribution. 

		Furthermore, it should be noticed that Barab\'{a}si--Albert model exhibits an asymptotic degree distribution
		that equals the Yule--Simon distribution in correspondence of a specific
		choice of the parameters and still presents power-law characteristics for more general choices of the parameters.
		The same happens also for other preferential attachment models. 

		Yule, Simon and Barab\'{a}si--Albert models share the preferential attachment
		paradigm that seems to play an important role in the explanation of the scale-freeness of real networks.
		However, the mathematical tools classically used in their analysis are different. This makes
		difficult to understand in which sense models producing very similar asymptotic
		distributions are actually related one another.
		Although often remarked and heuristically justified, no
		rigorous proofs exist clarifying conditions for such result.
		Different researchers from different disciplines, for example theoretical
		physicists and economists asked themselves about the relations
		between Simon, Barab\'{a}si--Albert, Yule and also some other models
		closely related to these first three (sometimes confused in the
		literature under one of the previous names). Partial studies in this direction
		exist but there is still a lack of clarifying rigorous results that would avoid
		errors and would facilitate the extension of the models. 

		The existing results refer to specific models and conditions but there is not a
		unitary approach to the problem. For instance, in \cite{PhysRevE.64.035104}, the authors compared the
		distribution of the number of occurrences of a different word in Simon
		model, when time goes to infinity, with the degree
		distribution in the Barab\'{a}si--Albert model, when the number of vertices
		goes to infinity. In \cite{SIMON1955}, an explanation relating the asymptotic
		distribution of the number of species in a random genus in Yule model and
		that of the number of different words in Simon model appears. More recently,
		following a heuristic argument, Simkin and Roychowdhury \cite{Simkin2011} gave a
		justification of the relation between Yule and Simon models.

		The aim of this paper is to study rigorously the relations between these
		three models. A fourth model, here named II-PA model (second preferential attachment model),
		will be discussed in order to better
		highlight the connections between Simon and Barab\'{a}si--Albert models. Also we include the Price model
		that predates the  Barab\'{a}si--Albert model, and is in fact the first model using a preferential attachment rule for networks.  
		
		The idea at the basis of our study is to make use of random graph processes theory to
		deal with all the considered discrete-time models and to include in this
		analysis also the continuous-time Yule model through the introduction of two suitable
		discrete-time processes converging to it. In this way we find a relationship between the discrete time models and the continuous time
		Yule model, which is easier to handle and extensively studied. Translating results from discrete models to their
		continuous counter-parts  is usually a strong method to analyze asymptotic properties.
		Thus, Theorems \ref{S-Yule} and \ref{S-Yule2}  provide an  easy tool for this.

		The random graph process approach was used by Barab\'{a}si and Albert to define their
		preferential attachment model of World Wide Web \cite{Barabasi1999}. At each discrete-time step a new vertex
		is added together with $m$ edges originating from it.  The end points
		of these edges are selected with probability proportional to the current degree of the vertices in the network.
		Simulations from this model show that the proportion of vertices with degree $k$ is $c_{m}k^{-\gamma }$,
		with $\gamma $ close to $3$ and $c_{m}>0$ independent
		of $k$. A mathematically rigorous study of this model was then performed
		by Bollob\'{a}s, Riordan, Spencer and Tusnady \cite{Bollobas2001} making use of random graph
		theory.
		The
		rigorous presentation of the model allowed the authors to prove that
		the proportion of vertices with degree $k$ converges in probability to $m(m+1)B\left( k,3\right)$
		as the number of vertices diverges, where $B\left( x,y\right)$
		is the Beta function. 

		Here we reconsider all the models of interest in a random graph process framework.
		In Section \ref{definizioni} we introduce the necessary notations and basic definitions.
		Then, in Section \ref{bbbb}, we present the four preferential attachment
		models of interest, i.e.\ Simon, II-PA, Price, Barab\'{a}si--Albert, and Yule models,
		through a mathematical description that makes use of the random graphs
		approach. Such a description allows us to highlight an aspect not always well
		underlined: the asymptotic distributions that in some cases coincide do not
		always refer to the same quantity. For instance, the Barab\'{a}si--Albert model
		describes the \emph{degree} of the vertices while II-PA considers the \emph{in-degree}. In
		Section \ref{bbbb} we also discuss the historical context and the list of available
		mathematical results for each model. The proposed point of view by means of random
		graphs processes then permits us to prove the novel results presented in Section \ref{princip}.
		The theorems described and proved there clarify the relations between the considered asymptotic
		distributions of the different models, specifying for which
		choice of the parameters these distributions coincide and when they are not
		related.

		In the concluding Section \ref{finale} we summarize the proved results and we illustrate
		with a diagram the cases in which the considered models are actually related.

	\section{Definitions and mathematical background}

		\label{definizioni}
		In this section we introduce some classical definitions, theorems and mathematical tools we will use in the rest of the paper. 

		Let us define a \emph{graph} $G=(V,E)$ as an ordered pair comprising a set of vertices $V$ with a set of edges or lines $E$ which are
		$2$-elements subsets of $V$, so $E\subseteq V\times V$.
		A graph $G$ is \emph{directed} if its edges are directed, i.e., if for every edge $(i,j)\in E$, $(i,j)\neq (j,i)$,
		otherwise $G$ is called an \emph{undirected} graph.  

		We  say  that $G$ is a \emph{random graph}, if it is a graph selected according with a probability distribution over a set of graphs,
		or it is determined by a stochastic process that describes the random evolution of the graph in time.
		A stochastic process generating a random graph is called a \emph{random graph process}.
		In other words, a random graph process is a family $(G^t)_{t\in \mathcal{T}}$ of random graphs (defined on a common probability space)
		where $t$ is interpreted as time and $\mathcal{T}$ can be either countable or uncountable. 

		A \emph{loop} is an edge that connects a vertex to itself.
		The \emph{in-degree} of a vertex $v$ at time $t$, denoted by $\vec{d}(v,t)$, is the number of incoming edges (incoming connections).
		Similarly, the \emph{degree} of a vertex $v$ at time $t$, denoted by $d(v,t)$,
		is the total number of incoming and outgoing edges at time $t$ (when an edge is a loop, it is counted twice).
		In this paper we also use the term \emph{directed loop} to indicate a loop that counts one to the in-degree. 
		
		The random graphs studied in this paper are random graph processes starting at time $t=0$, without any edge neither vertex,
		growing monotonically by adding at each discrete time step either a new vertex  or some directed edges between the vertices
		already present, according to some law $\bbP(v_i^t\longrightarrow v_j^t)=\bbP((i,j)\in G^t)$.

		We  focus here on the analysis of the  number of vertices with degree or in-degree $k$ at time $t$, which we
		denote by $N_{k,t}$ and $\vec N_{k,t}$, respectively. In particular  we are interested in the asymptotic degree
		or in-degree distribution of a random vertex, i.e., in the proportion $N_{k,t}/V_t$ or $\vec N_{k,t}/V_t$,  as $t$ goes to infinity,
		where $V_t$ denotes the total number of vertices at time $t$.
		We will add an upper index to $N_{k,t}$ or $\vec N_{k,t}$, for instance $\vec N_{k,t}^{\textup{Simon}}$, to
		indicate the process to which we refer, if necessary.

		Furthermore, we will make use of the following standard notation: for (deterministic) functions $f =f(t)$ and $g=g(t)$,
		we write $f =O(g)$ if $\lim_{t\to \infty} f/g$ is bounded, $f \sim g$ if $\lim_{t\to \infty} f/g=1$,
		and  $f = o(g)$ if $\lim_{t\to \infty} f/g=0$.

		One of the methods used  in the literature to study the asymptotic behavior  of $N_{k,t}/V_t$ or $\vec N_{k,t}/V_t$ is
		to prove that these random processes concentrate
		around their  expectations. In order to do this, the Azuma and Hoeffding inequality is applied, when possible (see also \cite{Durrett2006},
		page 93).

		\begin{lem}[Azuma and Hoeffding inequality \cite{JansonLuczakRucinski00}] Let $(X_t)_{t=0}^n$ be a martingale with $|X_s-X_{s-1}|\leq c$ for 				$1\leq s\leq t$ and $c$ a positive constant. Then
			\begin{align}
				\label{AzHoff}
				\bbP(|X_t-X_0|>x)\leq \exp(-x^2/2c^2t).
			\end{align}
		\end{lem}
		One of the first authors to use this approach in preferential attachment random graphs studies were
		Bollob\'as, \textit{et. al} in \cite{Bollobas2001}. Here we apply this approach to study different random graph processes.
		In Section \ref{bbbb} we illustrate this technique by analyzing the Simon model, reporting the corresponding computations
		for the Barab\'asi--Albert model.

	\section{Preliminaries: Preferential attachment models}

		\label{bbbb}
		As stressed in the introduction, a number of models that make use of ``preferential attachment''
		mechanisms are present in the literature.  Here we consider some of them rigorously introducing the corresponding random graph processes
		with the aim to allow a comparison of their features.
		To this aim, it helps to present the most known models using a common notation. We first discuss the case of discrete time preferential
		attachment models, specifically Simon and Barab\'asi--Albert models, and  some others inspired by Simon model,  the II-PA model
		(second preferential attachment model) and Price model, which will help us to understand  the relations between Simon
		and Barab\'asi--Albert models.
		Moreover, we also discuss a continuous time preferential attachment model, the Yule model, which is defined in terms of
		independent homogeneous linear birth processes.
		We rigorously prove that this model can be related with Simon, and hence Barab\'asi--Albert models.   

		The Barab\'asi--Albert model presented in \cite{Barabasi1999} omits some necessary
		details to be formulated in terms of a random graph process.
		Here we follow its description detailed as in Bollob\'as \textit {et.~al} \cite{Bollobas2001} where the rules for the growth of the random
		graph not mentioned in \cite{Barabasi1999} are given.
		Furthermore, in order to make easier the understanding of each model, we follow the same scheme for its presentation,
		eventually specifying the absence of some results when not yet available. 

		Our scheme considers:
		\begin{enumerate}
			\item The mathematical description of the associated graph structure and its growth law.
			\item The historical context motivating the first proposal of the model and some successive applications.
			\item Available results on the degree or in-degree distribution  with particular reference to  power law behavior.
				We collect both, theorems and simulation results.  
		\end{enumerate}

		\subsection{Simon model}
		
			\label{Simonmodel}
			\begin{enumerate}
				\item \textbf{Mathematical description:}
					The Simon model can be described as a random graph process in discrete time $(G_{\alpha}^t)_{t\geq1}$, so that $G_{\alpha}^t$
					is a directed graph which starts  at time $t=1$  with a single vertex $v_1$ and a directed loop.
					Then, given $G_{\alpha}^t$, one forms $G_{\alpha}^{t+1}$ by either  adding with probability $\alpha$ a new vertex
					$v_i$ with a directed loop, $i\leq t+1$,  or  adding with probability (1-$\alpha$) a directed edge  between the last added vertex
					$v$ and  $v_j$, $1\leq j\leq t$, where the probability of $v_j$ to be chosen is proportional to its in-degree, i.e., 
					\begin{align}
						\label{parule}
						\bbP(v \longrightarrow v_j) = (1-\alpha)\vec d(v_j,t)/t, \qquad 1\leq j\leq t.
					\end{align}
					In Figure \ref{figSimon} we illustrate the growth law of this graph.
	
					\tikzset{every loop/.style={min distance=10mm,in=0,out=60,looseness=10}}
	
					\begin{figure}[ht]
						\begin{center}
							\begin{tabular}{ccc}
								\hspace{1cm}&\hspace{1cm}
								\begin{tikzpicture}
									[scale=0.8]
									\tikzstyle{place}=[circle,thick,draw=blue!75,fill=blue!20,minimum size=6mm]
									\node[place] (foo)  {$v_1$};
									\draw (0,-1.2) node[below,font=\footnotesize] {$(a)$};
									\path[<-] (foo) edge  [loop above] ();
								\end{tikzpicture}
								\begin{tikzpicture}
									[scale=0.8]
									\tikzstyle{place}=[circle,thick,draw=blue!75,fill=blue!20,minimum size=6mm]
									\node[place] (v1)  {$v_1$};
									\node[place] (v2) at (2,0)  {$v_2$};
									\draw (1,-1.2) node[below,font=\footnotesize] {$(b)$};
									\path[<-] (v1) edge  [loop above] ();
									\path[->,every loop/.style={looseness=10}] (v1) edge  [in=120,out=60,loop] ();
									\path[<-]  (v2) edge  [loop above] ();        
									\path[->]  (v2) edge[auto, very near start]  (v1); 
								\end{tikzpicture}
								\begin{tikzpicture}
									[scale=0.8]
									\tikzstyle{place}=[circle,thick,draw=blue!75,fill=blue!20,minimum size=6mm]
									\node[place] (v1)  {$v_1$};
									\node[place] (v2) at (2,0)  {$v_2$};
									\node[circle,thick,draw=red!75,fill=red!20,minimum
									size=6mm] (v3) at (4,0)  {$v_3$};
									\draw (2,-1.2) node[below,font=\footnotesize] {$(c )$};
									\path[<-] (v1) edge  [loop above] ();
									\path[->,every loop/.style={looseness=10}] (v1) edge [in=120,out=60,loop] ();
									\path[->]  (v2) edge  [loop above] ();     
									\path[->]  (v2) edge[auto, very near start] (v1); 
									\path[->,red, dashed, very thick,every node/.style={font=\footnotesize},every loop/.style={looseness=10}]
									(v2) edge [in=120,out=60,loop] node [above right]
									{$(\ref{parule})$} ();
									\path[<-,red]  (v3) edge  [loop above] 
									node [above right] {$\alpha$} ();
									\path[->,red, dashed, very thick,every node/.style={font=\footnotesize}] (v2)
									edge [bend left] node [below left] {$(\ref{parule})$} (v1);
								\end{tikzpicture}
							\end{tabular}
						\end{center}
						\caption{\footnotesize{\label{figSimon}Construction of $(G_{\alpha}^t)_{t\geq1}$.
						(a) Begin at time $1$ with one single vertex and a directed loop.
						(b) Suppose some time has passed, in this case, the picture corresponds to a a realization of the process at time $t=4$.
						(c) Given $G_{\alpha}^4$ form $G_{\alpha}^{5}$ by either  adding with probability $\alpha$
						a new vertex $v_3$ with a directed loop,
						or  adding a directed edge with probability given by ($\ref{parule}$).}}
					\end{figure}
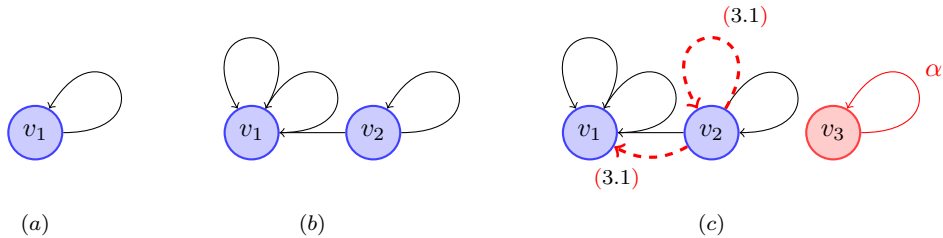
	
				\item \textbf{Historical context:}
					In \cite{SIMON1955}, Simon considered a model to describe the growth of a text that is being written such that a word is
					added at each time $t\geq1$. Different words  correspond to different vertices and repeated words to directed edges in the
					previous description. Simon introduced  the two following conditions:
					For $\alpha \in (0,1)$,
					\begin{enumerate}
						\item $\bbP[(t+1)\text{th word  has not yet appeared at time $t$}]=\alpha$
						\item $\bbP[(t+1)\text{th word  has  appeared $k$ times at time $t$}]=(1-\alpha)k \vec N_{k,t}/t$,
					\end{enumerate}
					where $\vec N_{k,t}$ is the number of  \textit{different} words that have appeared exactly $k$ times at time $t$,
					or  the number of vertices that have exactly $k$ incoming edges (i.e.\ in-degree $k$) at time $t$ in $G_{\alpha}^t$.
					Thus, at time $t+1$  either with  probability $\alpha$ a new word appears (i.e.,  a new vertex $v_i$, $i\leq t+1$,
					with a directed loop appears),  or with probability $(1-\alpha)$ the word is not new, and if it has  appeared $k$ times
					at time $t$,  a directed  edge is added.  The starting point of this edge is the last vertex that has appeared
					in $G_{\alpha}^t$,  while its  end point is selected with probability (\ref{parule}) that corresponds in this case to $k/t$. 
	
				\item \textbf{Available results:}
					Simon was interested  in getting  results for the proportion  of vertices that have exactly  in-degree $k$,
					with respect to the total number of vertices $V_t$ at time $t$.  Thus,  he proved asymptotic results for
					$\bbE \vec N_{1,t}/\bbE V_t$ as $t\longrightarrow\infty$.  
	
					Next, we will give a brief synopsis  of the computations made by Simon in \cite{SIMON1955}. 
					The idea is to condition on what has happened until time $t$ and compute the expected value at time $t+1$.
					For $k=1$ it holds 
					\begin{align}
						\label{Sk1}
						\bbE \vec N_{1,t+1}
						&= \alpha+\Big(1-\frac{(1-\alpha)}{t} \Big)\bbE \vec N_{1,t},
					\end{align}
					and, for $k>1$,
					\begin{align}
						\label{Sk} 
						\bbE \vec N_{k,t+1}
						&= \frac{(1-\alpha)}{t}\big[ (k-1)\bbE \vec N_{k-1,t}-k\bbE \vec N_{k,t}\big]+\bbE \vec N_{k,t}.
					\end{align}
	
					Simon  solved  (\ref{Sk1}) and (\ref{Sk}) (see also \cite{Durrett2006}, pages 98--99) to get, as 
					$t\longrightarrow\infty$,
					\begin{align}
						\label{p1} 
						\frac{\bbE \vec N_{1,t}}{t}\longrightarrow\frac{\alpha}{2-\alpha},
					\end{align} 
					and for $k>1$,
					\begin{align}
						\label{pk} 
						\frac{\bbE \vec N_{k,t}}{t}\longrightarrow\frac{\alpha}{1-\alpha}\frac{\Gamma(k)\Gamma
						\Big(1+\frac{1}{1-\alpha}\Big)}{\Gamma\Big(k+1+\frac{1}{1-\alpha}\Big)},
					\end{align}
					where $\Gamma$ is the  gamma function.
	
					Observe now that the number of vertices appeared until time $t$, $V_t\sim \text{Bin}(t,\alpha)$, so $\bbE V_t=\alpha t$.
					Hence, using this and (\ref{p1}) and (\ref{pk}) for $k=1$,
					\begin{align}
						\label{e1} 
						\frac{\bbE \vec N_{1,t}}{\bbE V_t}\longrightarrow\frac{1}{2-\alpha},
					\end{align} 
					and for $k>1$,
					\begin{align}
						\label{ek} 
						\frac{\bbE \vec N_{k,t}}{\bbE V_t}\longrightarrow\frac{1}{1-\alpha}\frac{\Gamma(k)\Gamma\Big(1
						+\frac{1}{1-\alpha}\Big)}{\Gamma\Big(k+1+\frac{1}{1-\alpha}\Big)}
						=\frac{1}{1-\alpha}B\Big(k,1+\frac{1}{1-\alpha}\Big),
					\end{align}
					as 
					$t\longrightarrow\infty$, where $B(x,y)$ is the Beta function.
	 
					Now,  let $\mathcal{G}_r$ and $\mathcal{G}_s$ denote the $\sigma$-fields generated by the appearance of directed edges up
					to time $r$ and $s$ respectively,  $r\leq s\leq t$. Since $\bbE\big[\bbE(\vec N_{k,t}|\mathcal{G}_s)
					\mid \mathcal{G}_r\big]=\bbE(\vec N_{k,t}|\mathcal{G}_r)$, then,  $Z_s^{\textup{Simon}} = \bbE(\vec N_{k,t}|\mathcal{G}_s)$
					is a martingale, such that  $Z_t^{\textup{Simon}}=\vec N_{k,t}$ and $Z_0^{\textup{Simon}}=\bbE\vec N_{k,t}$.
					Furthermore,  observe that  at each unit of time, say $s$,  either  a new vertex appears or  the last one added,
					is  attaching to another existing vertex  $v_j$, $j\leq s$, but note  this  does not effect the in-degree
					of $v\neq v_j$, or the probabilities these vertices will choose later, so it yields that
					$|Z_s^{\textup{Simon}} -Z_{s-1}^{\textup{Simon}}| \leq 1$.  
					Then,  it is possible to use  Azuma--Hoeffding's inequality (\ref{AzHoff}), and obtain that for every
					$\epsilon_t\gg t^{-1/2}$ (for example take $\epsilon_t=\sqrt{\ln t/t}$),
					\begin{align}
						\label{AzumaSimon} 
						\bbP\Big(\Big|\frac{\vec N_{k,t}}{t}-\frac{\bbE\vec N_{k,t}}{t}\Big|\geq\epsilon_t\Big)\leq
						\exp\Big(-\frac{(t \epsilon_t)^2}{2t}\Big)\longrightarrow 0.
					\end{align}
					Now, using Chebyschev's inequality, for every  $\varepsilon_t>0$, such that $t\varepsilon_t^2\longrightarrow\infty$
					as $t\longrightarrow\infty$,
	    				\begin{align}
	    					\label{chebyshevVt} 
						\bbP\Big(\Big|\frac{V_t}{ t}-\frac{\bbE V_t}{t}\Big|\geq\varepsilon_t\Big)
						\leq\frac{t\alpha(1-\alpha)}{t^2\varepsilon_t^2}\longrightarrow0.
					\end{align} 
					Hence, by  (\ref{AzumaSimon}) and  (\ref{chebyshevVt}), $\vec N_{k,t}/t\longrightarrow \bbE \vec N_{k,t}/t$
					and $V_t/ t\longrightarrow \bbE V_t/t$ in probability. 
	
					Finally,  since $\bbE \vec N_{k,t}/t$ and $\bbE V_t/t$ converge as $t$ goes to infinity to the constant values
					(\ref{p1}) and (\ref{pk}) respectively,  and because $V_t$ is a random variable with binomial distribution,
					$\text{Bin}(t,\alpha)$,  then by properties of convergence in probability we obtain that 
					\begin{align}
						\label{simondist}
						\frac{\vec N_{k,t}}{V_t}\longrightarrow \frac{1}{1-\alpha}\frac{\Gamma(k)\Gamma\Big(1+\frac{1}{1-\alpha}
						\Big)}{\Gamma\Big(k+1+\frac{1}{1-\alpha}\Big)}=\frac{1}{1-\alpha}B\Big(k,1+\frac{1}{1-\alpha}\Big),
					\end{align}
					in probability.
			\end{enumerate}

		\subsection{II-PA model (second preferential attachment model)}

			\label{IIPAMDesc}
			\begin{enumerate} 
				\item \textbf{Mathematical description:}
					In \cite{Newman2005} a different model is analyzed. In that paper it is called Yule model and described in discrete time.
					The model is defined also as a preferential attachment model but  in this case at each time step $n$ a new vertex is added with
					exactly $m+1$ directed edges, $m\in\bbZ^+$. These edges start from the new vertex and are directed towards any of the previously
					existing vertices according to a preferential attachment rule.
					To define formally a random graph process, we can think for a moment at an increasing time rescaled by
					$1/(m+1)$ so that at each unit of time $n$, $m+1$ scaling time steps happen. 
					Let  $(\tilde{G}_m^t)_{t\geq 1}$ be a random graph process  such that for all $n\in \bbZ^+\cup\{0\}$,
					\begin{enumerate}
						\item  at time $t=n(m+1)+1$  add a new vertex $v_{n+1}$ with a directed loop (it does count one for the in-degree), and
						\item  for $i=2,\dots,m+1$ at each time $t=n(m+1)+i$  add a directed edge from $v_{n+1}$ to  $v_j$, $1\leq j\leq n+1$, with 								probability
							\begin{align}
								\label{IIPAM}
								\mathbb{P}(v_{n+1} \longrightarrow v_j) = \frac{\vec d(v_j,t-1)}{t-1} .
							\end{align} 
							Note then that  $(\tilde{G}_m^t)_{t\geq 1}$ starts at time $t=1$  with a single vertex and one directed loop.
					\end{enumerate}
					In Figure \ref{II-PA-fig} we illustrate the growth law of this graph.
					
					\tikzset{every loop/.style={min distance=10mm,in=0,out=60,looseness=10}}

					\begin{figure}[ht]
						\begin{center}
							\begin{tabular}{ccc}
								\hspace{1cm}&\hspace{1cm}
								\begin{tikzpicture}
									[scale=0.8]
									\tikzstyle{place}=[circle,thick,draw=blue!75,fill=blue!20,minimum size=6mm]
									\node[place] (foo)  {$v_1$};
									\draw (0,-1.2) node[below,font=\footnotesize] {$(a)$};
									\path[<-] (foo) edge  [loop above] ();
								\end{tikzpicture}
								\begin{tikzpicture}
									[scale=0.8]
									\tikzstyle{place}=[circle,thick,draw=blue!75,fill=blue!20,minimum size=6mm]
									\node[place] (v1)  {$v_1$};
									\draw (0,-1.2) node[below,font=\footnotesize] {$(b)$};
									\path[<-] (v1) edge  [loop above] ();
									\path[->,every loop/.style={looseness=10}] (v1) edge  [in=120,out=60,loop] ();
									\path[->,every loop/.style={looseness=10}] (v1) edge  [in=180,out=120,loop] ();
								\end{tikzpicture}
								\begin{tikzpicture}
									[scale=0.8]
									\tikzstyle{place}=[circle,thick,draw=blue!75,fill=blue!20,minimum size=6mm]
									\node[place] (v1)  {$v_1$};
									\node[place] (v2) at (2,0)  {$v_2$};
									\draw (1,-1.2) node[below,font=\footnotesize] {$(c )$};
									\path[<-] (v1) edge  [loop above] ();
									\path[->,every loop/.style={looseness=10}] (v1) edge [in=120,out=60,loop] ();
									\path[->,every loop/.style={looseness=10}] (v1) edge  [in=180,out=120,loop] ();
									\path[->]  (v2) edge  [loop above] ();     
									\path[->,red, dashed, very thick,every loop/.style={looseness=10},every node/.style={font=\footnotesize}]
									(v2) edge [in=120,out=60,loop] node [above right]
									{$(\ref{IIPAM})$} ();
									\path[->,red, dashed, very thick,every node/.style={font=\footnotesize}] (v2)
									edge [bend left] node [below left] {$(\ref{IIPAM})$} (v1);
								\end{tikzpicture}
							\end{tabular}
						\end{center}
						\caption{\footnotesize{\label{II-PA-fig}Construction of $(\tilde{G}_{m}^t)_{t\geq1}$ for $m=2$.
						(a) Begin at time $1$ with one single vertex and a directed loop.
						(b) Suppose some time has passed, in this case, the picture corresponds to a a realization of the process at time $t=3$.
						Keep in mind that here $m=2$ and therefore $m=2$ directed edge are added to the graph by preferential attachment rule
						(but at this point the only possible choice is the vertex $v_1$).
						(c) Here time is $t=5$. A new vertex $v_2$ already appeared at time $t=4$ together with a directed loop.
						At time $5$ instead the first of the $m$ edges that must be added to the graph is chosen (red dashed directed edges)
						by means of the preferential attachment probabilities \eqref{IIPAM}.}}
					\end{figure}
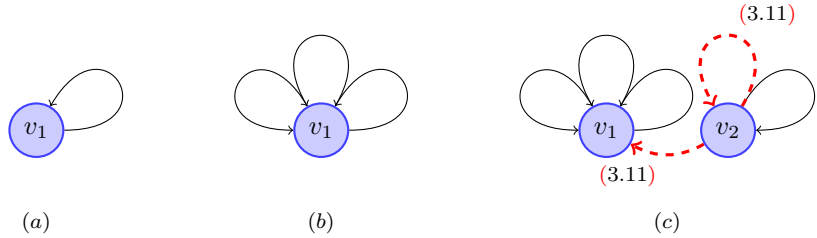
				
				\item \textbf{Historical context:}
					In \cite{Newman2005}, Newman  describes this model  in terms of genus and species as follows. 
					\begin{quotation}
						\noindent \textit{Species are added to genera by ``speciation'', the splitting of one species into two, $[\dots]$.
						If we assume that this happens at some stochastically constant rate, then it follows that a genus with $k$ species
						in it will gain new species at a rate proportional to $k$, since each of the $k$ species has the same chance per unit
						time of dividing in two. Let us further suppose that occasionally, say once every $m$ speciation events,
						the new species produced is, by chance, sufficiently different from the others in its genus as to be considered the
						founder member of an entire new genus. (To be clear, we define $m$ such that $m$ species are added to preexisting genera
						and then one species forms a new genus. So $m+1$ new species appear for each new genus and there are $m+1$
						species per genus on average.)} 
					\end{quotation}
					This description is  linked to the model proposed by Simon;
					the difference  is that the original Simon model does not fix  $m$ speciation events, instead it assumes
					that the number of speciation events is random and follows a probability distribution
					$\text{Geo}(\alpha)$, with $0<\alpha<1$. 

				\item \textbf{Available results:}
					Note that the number of vertices with in-degree equal to $k$ is equivalent to the number  of genera that have
					$k$ species, thus,  the number of vertices with in-degree equal to $k$ at  time  $t=n(m+1)$, corresponds to the number
					of genera that have $k$ species, when the number of genera is $n$.

					Let $\vec N_{k,t}$ be the number of vertices with in-degree equal to $k$ in $(\tilde{G}_m^t)_{t\geq 1}$.
					In \cite{Newman2005}  an heuristic analysis of the II-PA  model shows that  the proportion of vertices that have exactly
					in-degree $k$, with respect to the total number of vertices
					at time $t=n(m+1)$, is in the limit
					\begin{equation}
						\label{West}
						\lim_{n\longrightarrow\infty} \frac{\vec N_{k,t}}{n} =\frac{(1+1/m)\Gamma(k)\Gamma(2+1/m)}{\Gamma
						(k + 2 + 1/m)}=(1 + 1/m)B(k, 2 + 1/m),
					\end{equation} 
					We prove this  in Theorem \ref{newman} where the result is obtained with probability one.
			\end{enumerate}

		\subsection{Price model}

			\label{Price}
			\begin{enumerate} 
				\item \textbf{Mathematical description:}
					In \cite{Newmannew}, the Price model is described as a random graph process in discrete time
					$(\tilde{G}_m^n)_{n\geq 1}$, so that $\tilde{G}_m^n$ is a direct graph  and the process starts at time
					$n=1$ with a single vertex, $v_1$, and $M_1+k_0$ directed loops, where $k_0>0$ is constant and $M_1$ a random variable  with expectation $m$.  New vertices are continually added to the network, though not necessarily at a constant rate. Each added vertex has a certain out-degree,  and this out-degree is fixed permanently at the creation of the vertex. The out-degree may vary from one vertex to another, but the mean out-degree, which is denoted $m$, is a constant over time.
					Thus, given $\tilde{G}_m^n$ form $\tilde{G}_m^{n+1}$ by adding a new vertex $v_{n+1}$ with $k_0$ directed loops,
					and from it a random number of directed edges, $M_{n+1}$ to different old vertices with probability proportional
					to their in-degrees at time $n$, i.e.,  
					\begin{align}
						\label{Priceprob}
						\mathbb{P}(v_{n+1} \longrightarrow v_j\mid M_1=m_1,\dots,M_n=m_n) = \frac{\vec d(v_n,n)}{nk_0+\sum_{i=1}^{n} m_i}, \qquad 1\leq j\leq n,
					\end{align} 
					where  $M_1,\dots,M_{n+1}$  are taken  independent and identically distributed, with $\bbE(M_i)=m$,
					and $m$  a positive rational number. Note that  in this model, the update of the probabilities
					(\ref{Priceprob}) every single time an edge is added, is not taken into account.		
			
                \item \textbf{Historical context:} In \cite{342}, Price describes empirically the nature of the total world network
                	of scientific papers, and it is probably the first example of what is now called a scale-free network.
                	In \cite{343}, he formalizes a model giving rise to what he calls \textit{the cumulative advantage distribution}.
                	He finds a system of differential equations describing the process, and solves them under specific assumptions.
                	All the derivations are made for $k_0=1$.			
                		
                \item \textbf{Available results:} Let $\vec N_{k,n}$ be the number of vertices with in-degree equal to
                 	$k$ in $(\tilde{G}_m^n)_{n\geq 1}$. Newman \cite{Newmannew} analyzes this model by using the method of
                 	master-equations for the case $k_0=1$,  and  finds the same system as in \cite{Newman2005} for the analysis of the II-PA model.
                 	Thus, he obtains the same limit solution for  the proportion of vertices with in-degree $k$, as in the II-PA model, i.e., 
					\begin{equation}
						\label{Pricelimit}
						\lim_{n\longrightarrow\infty} \frac{\vec N_{k,n}}{n} =\frac{(1+1/m)\Gamma(k)\Gamma(2+1/m)}{\Gamma
						(k + 2 + 1/m)}=(1 + 1/m)B(k, 2 + 1/m).
					\end{equation} 
					A rigorous analysis of (\ref{Pricelimit}) can be made using Chebyschev's inequality and following the same lines as
					in the proof of Theorem \ref{newman} for the II-PA model (see Section \ref{PriceApprox}).   
				\end{enumerate}

		\subsection{Barab\'asi--Albert model}
		
			\label{BAMDesc} 
 			\begin{enumerate} 
				\item \textbf{Mathematical description:}
					In \cite{Bollobas2001}, Bollob\'as, Riordan, Spencer, and Tusnady  make the  Barab\'asi and Albert model
					precise in terms of a random graph process.  We follow their description in this paragraph. Add at each time step
					a new vertex with $m$, $m\in\bbZ^+$,  different directed edges.
					For the case $m=1$, let   $(G_1^t)_{t\geq1}$ be a random graph process  so that $G_1^t$ is a directed graph which
					starts at time $t=1$ with one vertex $v_1$ and one loop. Then, given  $G_1^{t}$ form $G_1^{t+1}$ by adding the vertex
					$v_{t+1}$ together with a single edge directed from $v_{t+1}$ to $v_j$, $1\leq j\leq t+1$,  with probability 
					\begin{align}
						\label{Bollobasm1}
						\mathbb{P}(v_{t+1}\longrightarrow v_j) =
						\begin{cases}
							\frac{d(v_j,t)}{2t+1}, & 1\leq j\leq t,\\
							\frac{1}{2t+1}, & j=t+1.
						\end{cases}
					\end{align}
					For $m>1$ define the process $(G_m^t)_{t\geq1}$ by running the process $(G_1^t)$ on the sequence  of imaginary vertices
					$v_1',v_2',\dots$, then form the graph $G_m^t$  from $G_1^{mt}$ by identifying the vertices
					$v_1',v_2',\dots,v_m'$ to form $v_1$,  $v_{m+1}',v_{m+2}'\dots,v_{2m}'$ to form $v_2$ and so on. 

					We can also define this model in a similar manner as we did for the  II-PA model. Thinking once more that the
					time increases  with a scaling of $1/(m+1)$, then  let us define  the process $(G_m^t)_{t\geq 1}$,
					such that for every $n\in \bbZ^+\cup\{0\}$,
					\begin{enumerate}
						\item  at time $t=n(m+1)+1$  add a new vertex $v_{n+1}$, 
						\item  for $i=2,\dots,m+1$ at each time $t=n(m+1)+i$  add an edge from $v_{n+1}$ to  $v$,
							where $v$ is chosen with
							\begin{align}
								\label{BA}
								\mathbb{P}(v_{n+1} \longrightarrow v) =
								\begin{cases}
									\frac{d(v,t-1)}{2(mn+i-1)-1}, &  v\neq v_{n+1}, \\
									\frac{d(v,t-1)+1}{2(mn+i-1)-1}, & v=v_{n+1}.
								\end{cases}
							\end{align}
							Observe that $(G_m^t)_{t\geq 1}$ starts at time $t=1$ just with a single vertex, without loops.
					\end{enumerate}

				\item \textbf{Historical context:}
					Barab\'asi and Albert, in \cite{Barabasi1999} proposed a random graph model of the growth of the world wide web,
					where the vertices represent sites or web pages, and the edges links between sites.  In this process the
					vertices are added to the graph one at a time and joined to a fixed number of earlier vertices, selected with
					probability proportional to their degree. This preferential attachment assumption is originated from the idea
					that a new site is more likely to join popular sites than disregarded sites. The model is described as follows.
					\begin{quotation}
						\noindent \textit{Starting with a small number $(m_0)$ of vertices,
						at every time step  add a new vertex with $m (\leq m_0)$
						edges that link the new vertex to $m$ different vertices already present in the system.
						To incorporate preferential attachment,  assume that the probability  that a new vertex will be connected
						to a vertex $i$ depends on the connectivity $k_i$ of that vertex, so it would be equal to $k_i/\sum_{j}k_j$.
						Thus, after $t$ steps the model leads to a random network with $t+m_0$ vertices and $mt$ edges.} 
					\end{quotation}

					To write a mathematical description of  the process given above it is necessary to clarify some details.
					First, since the model starts with $m_0$ vertices and none edges, then  the vertices degree are initially zero,
					so the probability  that the new vertex is connected to a vertex $i$, $1\leq i\leq m_0$, is not well defined.
					Second, to link the new vertex to $m$ different vertices already present, it should be necessary to repeat $m$
					times the experiment of choosing an old vertex, but the model does not say anything on changes of attachment
					probabilities at each time, i.e.\ it is not explained if the $m$ old vertices are simultaneously or sequentially chosen.
					These observations were made by  Bollob\'as, Riordan, Spencer, and Tusnady in \cite{Bollobas2001}, where after noted the
					problems in the Barab\'asi--Albert model, they give an exact definition of a random graph process that fits to that
					description.

				\item \textbf{Available results:}
					In \cite{Barabasi1999}, Barab\'asi and Albert  obtain through simulation  that after many time steps
					the proportion of vertices with degree $k$ obeys a power law $Ck^{-\gamma}$, where $C$ is a constant and
					$\gamma=2.9\pm 0.1$, and by a heuristic argument they suggest that $\gamma=3$.

					Let $N_{k,t}$ be the number of vertices with degree equal to $k$ in  $(G_m^t)_{t\geq 1}$.
					In \cite{Bollobas2001}  Bollob\'as, Riordan, Spencer, and Tusnady analyzed mathematically this model.
					Their first result is that, for $t=n(m+1)$, i.e., when the total number of vertices is $n$,
					and  $m\leq k\leq m+n^{1/15}$ (the bound $k\leq m+n^{1/15}$ is chosen to make the proof as easy as possible),
					\begin{align*}
						\frac{\bbE N_{k,t}}{n} \sim \frac{2m(m+1)}{k(k+1)(k+2)}=\alpha_{k},
					\end{align*}
					uniformly in $k$. 

					The authors consider $\mathcal{F}_s$, the $\sigma$-field generated by the appearance of directed edges up to time $s$,
					$s\leq t$, and define  $Z_s = \bbE(N_{k,t}|F_s)$ and see it is a martingale satisfying $|Z_s -Z_{s-1}| \leq 2$,
					$Z_t=N_{k,t}$ and $Z_0=\bbE N_{k,t}$, $k=1,2,\dots$ (at time $t=0$ the random graph is the empty graph).  
					Using Azuma--Hoeffding inequality (\ref{AzHoff}) they obtain that
					\begin{align*}
						\mathbb{P}\Big(\Big|\frac{N_{k,t}}{n}-\frac{\bbE N_{k,t}}{n}\Big|\geq \sqrt{\ln t/t}\Big)\leq
						\exp\Big(-\frac{\ln t}{8}\Big)\longrightarrow0,
					\end{align*} 
					as $t$ goes to infinity. Hence, it follows that, for every $k$ in the range $m\leq k\leq m+n^{1/15}$,
					\begin{align*}
						\frac{N_{k,t}}{n}\longrightarrow \alpha_{k},
					\end{align*}
					in probability. Thus, the proportion  of vertices with degree $k$, 
					\begin{equation}
						\label{Bollobas} 
						\frac{N_{k,t}}{n}\longrightarrow m(m+1)B(k,3)
					\end{equation}
					in probability as $t\longrightarrow\infty$. Note that
					\begin{align*}
						\frac{2m(m+1)}{k(k+1)(k+2)}=m(m+1)B(k,3).
					\end{align*}
					Furthermore, since the Beta function  satisfies the asymptotics $B(x,y)\longrightarrow x^{-y}$ for $x$ large enough,
					then $N_{k,t}/n\sim m(m+1)k^{-3}$ as $k\longrightarrow\infty$  and obeys a power law for large values of $k$,
					with $\gamma=3$ as  Barab\'asi and Albert  suggested.
					Hence, it is proved mathematically that when vertices are added to the graph one at a time and joined to a fixed number
					of existing vertices selected with probability proportional to their degree, the degree distribution follows
					a power law behavior \textit{only} in the tail (for $k$ big enough), with an exponent $\gamma=3$.
					A second proof of this result is given in \cite{RemcoBook} (see Theorem 8.2).
			\end{enumerate}
		
		\subsection{Yule model}
			
			\label{YuleDef}
			Differently from the previous models, this model evolves in continuous time. We do not describe this model in terms of
			random graph processes, however in subsection \ref{Relation2}  we discuss its relation with Simon model and conclude that
			the Yule model can be interpreted as a continuous time limit of Simon model, a model with a random graph interpretation.
			
			\begin{enumerate}
				\item \textbf{Mathematical description:}
					In the description of the Yule model we use $T$  to denote continuous time, instead of $t$ that denotes  discrete time,
					i.e.\ $T\in\mathbb{R}^+\cup\{0\}$ and $t\in\mathbb{Z}^+$.

					Consider a population starting at time $T=0$ with one individual. As time increases, individuals may give birth to
					new individuals independent of each other  at a constant rate $\lambda>0$, i.e., during any short time interval of
					length $h$ each member has probability $\lambda h+o(h)$ to create an offspring. Since there is no interaction among the
					individuals, then if  at epoch $T$ the population size is $k$,  the probability that an increase takes place at some
					time between $T$ and $T+h$ equals $k\lambda h+o(h)$.
					Formally, let $N(T)$ be the number of individuals at time $T$ with $N(0)=1$, then  if  $N(T)=k$, $k\geq1$, the probability of a
					new birth in $(T,T+h)$ is $k\lambda h+o(h)$, and the probability of more than one birth is $o(h)$, i.e.,  
					\begin{align*}
						\bbP(N(T+h)=k+\ell\mid N(T)=k)=
						\begin{cases}
							k\lambda h+o(h), & \ell=1,\\
							o(h), & \ell>1,\\
							1-k\lambda h+o(h), & \ell=0.
						\end{cases}
					\end{align*}

					Thus, $\{N(T)\}_{T\geq0}$ is a pure birth process and with the initial condition $\bbP(N(0)=k)=\delta_{k,1}$;
					this linear birth process is called the  \textit{Yule process}.   

					Consider now two independent Yule processes, $\{N_{\beta}(T)\}_{T\geq0}$ and $\{N_{\lambda}(T)\}_{T\geq0}$,
					with parameters   $\beta>0$ and $\lambda>0$ respectively,  such that when a new
					individual appears in the process with parameter $\beta$, a new Yule process with parameter $\lambda$ starts.
					In a random graph context,  a Yule model can be characterized through Yule processes of different parameters
					as described in the following.
					The first Yule process denoted by $\{N_{\beta}(T)\}_{T\geq0}$, $\beta>0$, accounts for the growth of the number
					of vertices. As soon as the first vertex is created, a second Yule process, $\{N_{\lambda}(T)\}_{T\geq0}$, $\lambda>0$,
					starts describing the creation of in-links to the vertex. The evolution of the number of in-links for the successively
					created vertices, proceeds similarly. Specifically, for each of the subsequent created vertices, an independent
					copy of   $\{N_{\lambda}(T)\}_{T\geq0}$, modeling the appearance of the in-links is initiated. 
 
					Let us define   $Y_0=0$ and for $k\geq 1$,
					\begin{align*}
						Y_k=\inf\{T \colon N_{\lambda}(T)=k+1\},
					\end{align*}
					so that $Y_k$ is the time of the $k$th birth, and   $W^*_k=Y_k-Y_{k-1}$ is the waiting time between the $(k-1)$th and the
					$k$th birth.  In a  Yule process it is well-known  that the waiting times  $W^*_k$, $k\geq 1$, are independent,
					each exponentially distributed with parameter $\lambda k$.

					Conversely, it is possible to reconstruct $\{N_{\lambda}(T)\}_{T\geq 0}$ from the knowledge of the $W^*_j$, $j\geq 1$,
					by defining
					\begin{align}
						\label{lastconst}
						Y_k=\sum_{j=1}^{k}W^*_j, \qquad \quad N_{\lambda}(T)=\min\{k \colon Y_{k}> T\}.
					\end{align}
					Thus if the $W^*_j$ are independently  distributed exponential random variables,
					of parameter $\lambda j$, then  $\{N_{\lambda}(T)\}_{T\geq 0}$ is a Yule process of parameter $\lambda$.

				\item \textbf{Historical context:}
					Yule in \cite{Yule1925} observed that the distribution of species per genus in the evolution of a biological population
					typically presents a power law behavior, thus, he proposed a stochastic model to fit these data.
					In the original paper \cite{Yule1925} the process is described as follows:
					\begin{quotation}
						\noindent \textit{Let the chance of a species ``throwing'' a specific mutation, i.e., a new species of the same genus,
						in some small assigned interval of time be $p$, and suppose the interval so small that $p^2$ may be ignored compared
						with $p$. Then, putting aside generic mutations altogether for the present, if we start  with $N$ prime species
						of different genera, at the end of the interval we will have $N(1-p)$  which remain monotype and $Np$ genera of two species.
						The new species as well as the old can now throw specific mutation.}
					\end{quotation}						
					Yule proceeded to the limit, taking the time interval $\Delta T$ as indefinitely but the number of such intervals
					$n$ as large, so that $n\Delta T=T$ is finite, and he wrote $p=\lambda\Delta T=\lambda T/n$.
					Yule not only studied this process.   In \cite{Yule1925}, he furthermore  constructed a model of evolution by considering
					two independent Yule processes, one for species with a constant rate $\lambda>0$ and the other for new genera (each of them
					composed by a single species) created at a constant rate $\beta>0$. In other words,  at time $T=0$ the process starts
					with a single genus composed by a single species.  As time goes on,  new genera (each composed by a single species) develop
					as  a Yule process of parameter $\beta$, and simultaneously and independently new species evolve as a Yule process with rate
					$\lambda$. Furthermore,  since a new genus appears with a single species, then each time a  genus births,
					a Yule process with rate $\lambda$ starts.    
 
				\item \textbf{Available results:} 
					Let $N_g(T)$ and $N_s(T)$, $T\geq0$, be the counting processes measuring the number of genera and species created until time
					$T$, respectively.  It is well-known that the probability distribution of the number of individuals in a Yule process with
					parameter $\lambda$ is geometric, $\text{Geo}(e^{-\lambda T})$. Thus, the distribution of the number of species  $N_s(T)$
					in a genus during the interval of time $[0,T]$ is
					\begin{equation}
						\label{distNt}
						\bbP(N_s(T)=k)=e^{-\lambda T}(1-e^{-\lambda T})^{k-1}, \qquad k\geq1,\: T\geq0. 
					\end{equation} 
					On the other hand, it is also known that by conditioning on the number of genera present
					at time $T$, the random instants at which creation of novel genera occurs are distributed as the order statistics of $iid$
					random variables with distribution function
					\begin{equation}
						\label{orderstat}
						\bbP(\mathcal{T}\leq\tau)=\frac{e^{\beta\tau}-1}{e^{\beta t}-1},\qquad 0\leq\tau\leq t
					\end{equation}
					(see \cite{LaskiPolito14} and the references therein).
					The authors in \cite{LaskiPolito14} take into account that the  homogeneous linear pure birth process lies in the class
					of the so-called processes with the order statistic property, see \cite{Neuts}, and use  \cite{Crump,Feigin,Puri}
					and \cite{Thompson} to get  (\ref{orderstat}).

					Thus, let $\mathcal{N}_T$ be the size of a genus chosen uniformly at random at time $T$. Then,
					\begin{align}
						\bbP(\mathcal{N}_T=k) & = \int_0^T\bbP(N_s(T)=k\mid N_s(\tau)=1)\bbP(\mathcal{T}\in d\tau)\nonumber\\
						&= \int_0^T e^{-\lambda(T-\tau)}(1-e^{-\lambda(T-\tau)})^{k-1}\beta\frac{e^{\beta\tau}}{e^{\beta T}-1}d\tau\nonumber\\
						&= \frac{\beta}{1-e^{-\beta T}}\int_0^T e^{-\beta y}e^{-\lambda y}(1-e^{-\lambda y})^{k-1} dy.
					\end{align}
					The interest now is in the limit behavior when $T\longrightarrow\infty$:
					\begin{align}
						\lim_{T\longrightarrow\infty} \bbP(\mathcal{N}_T=k) = \beta\int_0^{\infty}
						e^{-\beta y}e^{-\lambda y}(1-e^{-\lambda y})^{k-1} dy.
					\end{align}
					Letting $\rho=\beta/\lambda$ it is possible to recognize the integral as a beta integral to obtain
					(see \cite{Yule1925}, page 39)
					\begin{align}
						\lim_{T\longrightarrow\infty}\bbP(\mathcal{N}_T=k)=\rho\frac{\Gamma(k)\Gamma(1+\rho)}{\Gamma(k+1+\rho)}
						=\rho B(k,1+\rho),\qquad k\geq1.
					\end{align}
			\end{enumerate} 

	\section{Main Results}

		\label{princip}
		\subsection{Relations between the four discrete-time models}

			In \cite{PhysRevE.64.035104}, Bornholdt and Ebel pointed out that the asymptotic power law of the Barab\'asi--Albert
			model with $m=1$ coincides with that of the Simon model characterized by $\alpha=1/2$ (see (\ref{simondist}) and (\ref{Bollobas})).
			From this observation they suggested that for $m=1$ Barab\'asi--Albert model could be mapped to the subclass of Simon models
			with $\alpha=1/2$.
 
			We think that  (\ref{simondist}) and (\ref{Bollobas}) cannot be compared since they  give the asymptotic of different random variables.
			Indeed the first refers to the in-degree and the second to the degree distribution.   However we do believe that may exist a
			relation between these two models, which needs to be explained. 
 
			In this section we discuss rigorous arguments that allow us to clarify the relation between  Simon and Barab\'asi--Albert models.
			To this aim we make use also  of  the II-PA model.  We first relate Barab\'asi--Albert and II-PA models and then II-PA and Simon models.
			This double step is made necessary by the different quantities described by these models.
			
			Now, we are ready to formulate our results. 

			\begin{teo}
				\label{m=1}
				Let $m=1$. Then, the  in-degree distribution of the II-PA model  and the degree distribution of the  Barab\'asi--Albert model
				at time $t$, $t\geq2$, are  the same, i.e., if at time $t$ there are $n$ vertices in the processes,
				then for any $k\in\bbZ^+$,
				\begin{align*}
					\frac{\vec N_{k,t}^{\textup{II-PA}}}{n}=\frac{N_{k,t}^{\textup{BA}}}{n},
				\end{align*}
				where  $\vec N_{k,t}^{\textup{II-PA}}$ and $N_{k,t}^{\textup{BA}}$ denote the number of vertices with
				in-degree and degree equal to $k$ in $(\tilde G_1^t)$
				and $(G_1^t)$ at time $t$, i.e., in the II-PA and Barab\'asi--Albert models, respectively. 
			\end{teo}
			\begin{proof}
				We follow  the mathematical description of the II-PA and Barab\'asi--Albert models in terms of
				the random graph processes $(\tilde{G}_m^t)_{t\geq 1}$ and $({G}_m^t)_{t\geq 1}$, presented in
				Sections \ref{IIPAMDesc} and \ref{BAMDesc}, respectively. Let us  divide each unit of time in two sub-units.
				At each instant of time $t=2n+1$ a new vertex $v_{n+1}$ is created in both models; in the  II-PA model this vertex is created
				together with a directed loop. Furthermore, at each time $t=2n+2=2(n+1)$ an edge (a directed edge in the II-PA model)
				is added from $v_{n+1}$ to $v_j$,  $j\leq n+1$, with probabilities given by (\ref{IIPAM}) and (\ref{BA}) for the
				II-PA and  Barab\'asi--Albert models, respectively.  Hence our thesis corresponds to show that (\ref{IIPAM})
				and (\ref{BA}) coincide under our hypotheses.

				We see that the denominator for both probabilities (\ref{IIPAM}) and (\ref{BA}) is $2n+1$, and although  the two numerators
				count different quantities, the in-degree for the II-PA and the degree for Barab\'asi--Albert models,
				their values also coincide. This is easy  to check when $v_j=v_{n+1}$ and  the directed edge created at time $t=2(n+1)$
				is to $v_{n+1}$.
				In fact  the numerators of (\ref{IIPAM}) and (\ref{BA}) become both one. Let us now show that the two numerators coincide
				also when $v_j \neq v_{n+1}$.

				Let us suppose  $v_j\neq v_{n+1}$, and let $t=2(n+1)-1=2n+1$. Observe that in the Barab\'asi--Albert model the degree of $v_j$
				at time $t=2n+1$,  $d(v_j,2n+1)$, $j\in\{1,2,\dots,n\}$,   is the sum of  the number of incoming edges from time
				$t=2j+1$ (when $v_{j+1}$ is added) to time  $2n+1$,   plus the degree corresponding to the edge added at time
				$t=2j$ from $v_j$, that is two if the edge was a loop and one otherwise.  On the other hand, in the II-PA model the in-degree
				of $v_j$ at time $t=2n+1$,  $\vec d(v_j,2n+1)$, $j\in\{1,2,\dots,n\}$ is  the sum of the number of incoming edges added in the
				interval of time $t \in [2j+1,2n+1]$ (so this part coincides with  Barab\'asi--Albert model), plus the in-degree corresponding
				to the directed edge added at time $t=2j$ from $v_j$.
				Thus,  if it is a directed loop to $v_j$,  the in-degree of $v_j$ at time $t=2j$ is two (since when $v_j$ appeared,
				it did together with a  directed loop),  otherwise the in-degree is one.
				This concludes the proof and  (\ref{IIPAM}) and (\ref{BA}) coincide.
			\end{proof}

			\begin{obs}
				The proof of Theorem \ref{m=1} enlightens the advantage given by the re-definition of existing models in terms of
				random graph processes. In particular this reading shows immediately that the two models can be related only when $m=1$.  
			\end{obs}

			\begin{obs}
				The  in-degree distribution of the II-PA model  and the degree distribution of the  Barab\'asi--Albert model  are different
				when $m>1$. In fact take for example $m=2$ and suppose the first directed edge from $v_{n+1}$ is not a loop,
				i.e., a vertex $v_j$, $1\leq j\leq n$ is chosen.  Then, at time $t=3n+1$, $\vec d(v_{n+1},3n+1)=1$ in the II-PA model, while
				$d(v_{n+1},3n+1)=2$ in the Barab\'asi--Albert model. Thus at time $t=3n+2$,
				(\ref{IIPAM}) and (\ref{BA}) are different because the corresponding numerators differ. 
			\end{obs}

			Next we discuss the relationship between Simon and the II-PA models, which allows us to relate
			Barab\'asi--Albert and Simon models. Before writing such a result, observe the following fact.
			Let $Y_i$ be a random variable that counts the number of direct edges originated in the Simon model
			by the $i$th vertex $v_i$, until the appearing of the $(i +1)$th vertex. Note that $Y_i$ follows a Geometric
			distribution with parameter $\alpha$. So, if $\alpha = 1/(m +1)$, then $\bbE Y_i = m$, and  that is the number
			of out-going links from a vertex in the the II-PA and Barab\'asi--Albert models.
			
			What we will establish in the following theorem is that the asymptotic in-degree distribution of the II-PA and Simon
			models coincide when $\alpha = 1/(m + 1)$. To do that, first we introduce the following definition.
				
			\begin{defn}
				We say that a vertex $v_i$   appears  ``complete'' when   it has appeared in the process together with all the directed edges
				originated from it.  Thus,  at time $t=n(m+1)$,  the II-PA  model has exactly $n$  ``complete'' vertices.
			\end{defn}
			Now we are ready to enunciate the theorem. 

			\begin{teo}
				\label{newman}
				Let $m\in \bbZ^+$ be fixed. If $(\tilde{G}_m^t)_{t\geq 1}$ is the random graph process defining
				the II-PA model, and $\vec N_{k,t}^{\textup{II-PA}}$ the number of vertices with in-degree equal  $k$, $k\geq1$,
				at time  $t$ in $(\tilde{G}_m^t)_{t\geq 1}$, 
				then, at time $t=n(m+1)$
				\begin{equation}
					\label{IIPam}
					\frac{\vec N_{k,t}^{\textup{II-PA}}}{n}\longrightarrow
					\frac{(1+1/m)\Gamma(k)\Gamma(2+1/m)}{\Gamma(k + 2 + 1/m)}
				\end{equation}
				almost surely as $n\longrightarrow\infty$.
			\end{teo}

			\begin{obs}
				Observe that by (\ref{IIPam}) and (\ref{simondist}) if $\alpha = 1/(m + 1)$ in the Simon model,
				then the asymptotic in-degree distribution of Simon and II-PA models coincide.  Moreover, from  Theorem
				\ref{m=1}, (\ref{IIPam}) and (\ref{simondist}), it follows that the asymptotic degree distribution
				of  Barab\'asi--Albert model coincides with the asymptotic in-degree distribution of  Simon model only when $m=1$ and
				$\alpha=1/(m+1)$, so that $\alpha=1/2$. We conjecture that some other properties of  Barab\'asi--Albert model when $m=1$,
				for example the diameter, should be also related with the analogous features of Simon model when $\alpha=1/2$.
			\end{obs}  

			\begin{obs}
				Observe that (\ref{IIPam}) coincides with  (\ref{West}). Thus, the previous theorem gives a rigorous formalization
				to the heuristic result in \cite{Newman2005}.
			\end{obs}
  
			\begin{obs}
				Theorem \ref{newman} can be compared with the recent model-free approach of Ostroumova, Ryabchenko and Samostav
				(see Section 3, Theorem 2 in \cite{12053015}, with $A=m/(m+1)$ and $B=0$). However, in that work a preferential attachment
				rule proportional to the degreeis considered, and  Theorem 2 in \cite{12053015}  makes use of the initial condition that the degree
				of an existing vertex should be at least equal to $m$. Instead, in the II-PA model, it is considered a preferential attachment
				rule proportional to the in-degree with the initial condition that the in-degree of an existing vertex should be at least equal to
				one. Therefore, the II-PA model does not fit into the general setup of  Theorem 2 in \cite{12053015} and we cannot directly apply
				it to get the result of  Theorem \ref{newman} given in this paper. We believe however that, following these new ideas,
				but considering the in-degree and the corresponding initial condition we can obtain the asymptotic in-degree distribution
				for the II-PA model. 
				
				However, in this paper we use the master equations approach for consistency with the theory used
				to study Simon model.
			\end{obs}
  		
			Before proving Theorem  \ref{newman} we need to prove the following lemmas.
			\begin{lem}
				\label{lema1}
				Let $r,s,t\in \bbZ^+$ and $b\in \bbR$ such that $|b/r|< 1$, then 
				\begin{align*}
					\prod_{r=s+1}^t \Big(1-\frac{b}{r}\Big)=\Big(\frac{s}{t}\Big)^b\Big(1+O\Big(\frac{t-s}{st}\Big)\Big).
				\end{align*}
			\end{lem} 
			\begin{proof}
				Since $|b/r|< 1$, then using Taylor expansion for $\ln(1-b/r)$ we get
				\begin{align*}
					\prod_{r=s+1}^t \Big(1-\frac{b}{r}\Big)=\exp\Big[\sum_{r=s+1}^t\left(\frac{-b}{r}+O\Big(\frac{b^2}{r^2}\Big)\right)\Big].
				\end{align*}
				Now, by Euler--Maclaurin it is possible to obtain  that (see \cite{IAAlg}) 
				\begin{enumerate}
					\item $\sum_{r=1}^{t}\frac{1}{r}=\ln t+1-\int_{1}^{t}\frac{y-\lfloor{y}\rfloor}{y^2}dy$,
					\item $\sum_{r=1}^{t}\frac{1}{r^2}=\frac{1}{t}-2\int_{1}^{r}\frac{\lfloor{y}\rfloor}{y^{3}}dy$.
				\end{enumerate} 
				Using these expressions and the fact that $y-1\leq\lfloor{y}\rfloor\leq y$,
				where $\lfloor y \rfloor$ indicates the integer part of $y$, we obtain
				\begin{align}
					\label{1/r}
					\ln t -\ln s -\frac{t-s}{st}& <\sum_{r=1}^{t}\frac{1}{r}< \ln t -\ln s ,\\
					\frac{t-s}{st}-\frac{t^2-s^2}{(st)^2}& <\sum_{r=1}^{t}\frac{1}{r^2}< \frac{t-s}{st},
				\end{align} 
				or,
				$\sum_{r=1}^{t}1/r= \ln t-\ln s-|\delta_1|,$ where $|\delta_1|<(t-s)/(st)$,
				and
				$\sum_{r=1}^{t}1/r^2= (t-s)/(st)-|\delta_2|$,  where  $|\delta_2|<(t^2-s^2)/(st)^2$. Thus,
				\begin{align*}
					\prod_{r=s+1}^t \Big(1-\frac{b}{r}\Big)&=\exp\Big[b\ln\Big(\frac{s}{t}\Big)+O\Big(\frac{t-s}{st}\Big)\Big]\\
					&=\Big(\frac{s}{t}\Big)^b\Big(1+O\Big(\frac{t-s}{st}\Big)\Big).
				\end{align*}
			\end{proof}

			\begin{lem}\label{lema2}
				Let $\vec N_{k,t}^{\textup{II-PA}}$ and  $\vec N(k,n)$ denote the number of vertices with in-degree equal
				$k$, $k\geq1$, at time $t$, and the number of vertices with in-degree $k$ when there are exactly $n$ complete vertices
				in the II-PA model, respectively.  Then,
				\begin{itemize}
					\item for $m=1$ and $k=1$,
						\begin{align}
							\label{nm1k1}
							\bbE\vec N(1,n+1) = \Big(1-\frac{1}{(m+1)(n+1)-1}\Big)+\Big(1-\frac{1}{(m+1)(n+1)-1}\Big)\bbE\vec N(1,n);
						\end{align}
					\item for $m>1$ and $k=1$,
						\begin{align}
							\label{nmk1}
							\bbE\vec N(1,n+1)=1+\Big(1-\frac{m}{(n+1)(m+1)-1}\Big)\bbE\vec N(1,n)+O\Big(\frac{1}{n}\Big);
						\end{align}
					\item for $m=1$ and $k\geq2$,
						\begin{align}
							\label{nm1k}
							\bbE\vec N(k,n+1)&=\frac{(k-1)\bbE\vec N(k-1,n)}{(n+1)(m+1)-1}+\Big(1-\frac{k}{(n+1)(m+1)-1}\Big)
							\bbE\vec N(k,n)\nonumber\\
							&=\frac{(k-1)\bbE\vec N(k-1,n)}{2(n+1)-1}+\Big(1-\frac{k}{2(n+1)-1}\Big)\bbE\vec N(k,n);
						\end{align}
					\item for $m>1$ and $k\geq2$,
						\begin{align}
							\label{nmk}
							\bbE\vec N(k,n+1)= \frac{(k-1)m\bbE\vec N(k-1,n)}{(n+1)(m+1)-1}+\Big(1-\frac{km}{(n+1)(m+1)-1}\Big)\bbE\vec
							N(k,n)+O\Big(\frac{k}{n}\Big).
						\end{align}
				\end{itemize}
			\end{lem}

			\begin{proof}
				Let $m=1$ and $k=1$ we start at time $t=(m+1)n=2n$, i.e., when there are exactly $n$ complete vertices.
				To see what happens when exactly $(n+1)$ complete vertices appear, we need to check what happens in two steps of the process,
				at time $2n+1$, when deterministically appears a new vertex with a directed loop, and at time $2n+2=2(n+1)$, when a new
				directed edge is added by preferential attachment, and the last vertex added becomes complete. Thus, conditioning on what
				happens until time $t+1$, we have 
				\begin{align}
					\label{m1k1}
					\bbE(\vec N_{1,t+2}^{\textup{II-PA}}) & = \bbE\Big[(\vec N_{1,t}^{\textup{II-PA}}+1)\Big(1-\frac{\vec
					N_{1,t}^{\textup{II-PA}}+1}{t+1}\Big)+\vec N_{1,t}^{\textup{II-PA}}\Big(\frac{\vec N_{1,t}^{\textup{II-PA}}+1}{t+1}\Big)
					\Big]\nonumber\\
					& = \Big(1-\frac{1}{t+1}\Big)+\Big(1-\frac{1}{t+1}\Big)\bbE(\vec N_{1,t}^{\textup{II-PA}}).
				\end{align}
				Thus, if  $\vec N(k,n)$ denotes  the number of vertices with in-degree $k$ when there are exactly $n$ complete vertices
				in the process, then we can write the previous equation as (\ref{nm1k1}).

				Let $m>1$ and $k=1$. Now we need a bit  more attention, since we have to consider two different situations,
				when $t$  is multiple of $(m+1)$ and when $t$ is not. In the first situation $t$ has the form  $t=n(m+1)$, so  we are in
				the instant of time when there are exactly $n$ complete vertices, and as we did above, to see what happens later we check
				what happens in the two subsequent steps of the process, at time $n(m+1)+1$ when a deterministic event happens,
				a new vertex with a directed loop appears, and at time $n(m+1)+2$ when something probabilistic happens, a new directed
				edge is added by preferential attachment.  In the first case equation (\ref{m1k1}) still holds. In the second situation observe
				that if $m>1$,  in order to see complete the  vertex added at time $t=n(m+1)+1$, we have to check what happens
				from $n(m+1)+1$ until $n(m+1)+(m+1)=(n+1)(m+1)$, when this vertex becomes complete.
				Thus, when $t$ is not multiple of $(m+1)$ we have the following equation.
				\begin{align}
					\label{mk1}
					\bbE(\vec N_{1,t+1}^{\textup{II-PA}}) & = \bbE\Big[\vec N_{1,t}^{\textup{II-PA}}
					\Big(1-\frac{\vec N_{1,t}^{\textup{II-PA}}}{t}\Big)+(\vec N_{1,t}^{\textup{II-PA}}-1)
					\Big(\frac{\vec N_{1,t}^{\textup{II-PA}}}{t}\Big)\Big]\nonumber\\
					& = \Big(1-\frac{1}{t}\Big)\bbE(\vec N_{1,t}^{\textup{II-PA}}).
				\end{align}
   
				Now, we may use simultaneously (\ref{m1k1}) and (\ref{mk1}) to get the corresponding equation of what happens in
				$(m+1)$ steps of the process.  We start at time $t=(n+1)(m+1)-1$, so at time $t+1$ the process will have exactly $(n+1)$
				complete vertices, and since $t$ is not multiple of $(m+1)$, we need to begin using (\ref{mk1}) $(m-1)$ times, and then use
				(\ref{m1k1}). Iterating $m$ times, we obtain
				\begin{align}
					\bbE(\vec N_{1,t+1}^{\textup{II-PA}}) & = \Big[1+E(\vec N_{1,t-m}^{\textup{II-PA}})\Big]
					\prod_{j=0}^{m-1}\Big(1-\frac{1}{t-j}\Big)\nonumber\\
					& = \Big[1+E(\vec N_{1,t-m}^{\textup{II-PA}})\Big]\prod_{r=t-(m-1)}^{t}\Big(1-\frac{1}{r}\Big)\nonumber\\
					& = \Big(1-\frac{m}{t}+O\Big(\frac{1}{t^2}\Big)\Big)\Big[1+E(\vec N_{1,t-m}^{\textup{II-PA}})\Big],
				\end{align}
				where  we have used in the the last two steps that  $r=t-j$ and  Lemma \ref{lema1}.
				Finally, using the notation $\vec N(1,n)$, and since $\vec N(1,n)/n\leq1$,  we get (\ref{nmk1}).

				The cases $k=2$ and $k>2$ need to be  first considered separately, and in each of these we need to analyze when
				$m=1$ and when $m>1$. Then we will show that the equations for $k=2$ and $k>2$ admit a general form,
				that include the cases $k\geq2$. 

				Let $m=1$.  Analogously as we did when $m=1$ and $k=1$,  consider the  time $t=n(m+1)$, i.e., when there are exactly
				$n$ complete vertices. To account for what happens until when $(n+1)$ complete vertices appear, we need to recognize
				two steps of the process. Indeed,
				\begin{align}
					\label{m1k2}
					\bbE(\vec N_{2,t+2}^{\textup{II-PA}})= {} &\bbE\Big[(\vec N_{2,t}^{\textup{II-PA}}+1)
					\Big(\frac{\vec N_{1,t}^{\textup{II-PA}}+1}{t+1}\Big)+(\vec N_{2,t}^{\textup{II-PA}}-1)
					\frac{2\vec N_{2,t}^{\textup{II-PA}}}{t+1}\nonumber\\
					& + \vec N_{2,t}^{\textup{II-PA}}\Big(1-\frac{\vec N_{1,t}^{\textup{II-PA}}+1+2
					\vec N_{2,t}^{\textup{II-PA}}}{t+1}\Big)\Big]\nonumber\\
					= {} & \frac{\bbE(\vec N_{1,t}^{\textup{II-PA}})+1}{t+1}+\Big(1-\frac{2}{t+1}\Big)\bbE(\vec N_{2,t}^{\textup{II-PA}}),
				\end{align}
				and for $k>2$,
				\begin{align}
					\label{m1k}
					\bbE(\vec N_{k,t+2}^{\textup{II-PA}}) = {} & \bbE\Big[(\vec N_{k,t}^{\textup{II-PA}}+1)
					\Big(\frac{(k-1)\vec N_{k-1,t}^{\textup{II-PA}}}{t+1}\Big)+(\vec N_{k,t}^{\textup{II-PA}}-1)
					\frac{k\vec N_{k,t}^{\textup{II-PA}}}{t+1}\nonumber\\
					& + \vec N_{k,t}^{\textup{II-PA}}\Big(1-\frac{(k-1)\vec N_{k-1,t}^{\textup{II-PA}}
					+k\vec N_{k,t}^{\textup{II-PA}}}{t+1}\Big)\Big]\nonumber\\
					= {} & \frac{(k-1)\bbE(\vec N_{k-1,t+1}^{\textup{II-PA}})}{t+1}
					+\Big(1-\frac{k}{t+1}\Big)\bbE(\vec N_{k,t+1}^{\textup{II-PA}}).
				\end{align}
				Note that in the last line of (\ref{m1k2}) and (\ref{m1k}),  we have replaced $\vec N_{k,t+1}^{\textup{II-PA}}$
				with $\vec N_{k,t}^{\textup{II-PA}}$. In fact if  $t=n(m+1)$, then at time $t+1$ the process just adds deterministically
				a new vertex with in-degree one, so when $k\geq2$, $\vec N_{k,t+1}^{\textup{II-PA}}=\vec N_{k,t}^{\textup{II-PA}}$,
				as well as $\vec N_{1,t+1}^{\textup{II-PA}}=\vec N_{1,t}^{\textup{II-PA}}+1$. Using this observation we can express
				(\ref{m1k2}) and (\ref{m1k}) as a single equation holding for $k\geq2$ and $m=1$.
				Using  the notation $\vec N(1,n)$  it can be written as (\ref{nm1k}).

				Let $m>1$. once more we need to consider when $t$ is multiple of $(m+1)$, and when it is not.
				When $t=n(m+1)$ we obtain again (\ref{m1k2}) and (\ref{m1k}) for $k=2$ and $k>2$, respectively, while if $t$ is not multiple of
				$(m+1)$ and $k\geq2$ it holds
				\begin{align}
					\label{mk}
					\bbE(\vec N_{k,t+1}^{\textup{II-PA}})= {} & \bbE\Big[(\vec N_{k,t}^{\textup{II-PA}}
					+1)\Big(\frac{(k-1)\vec N_{k-1,t}^{\textup{II-PA}}}{t}\Big)
					+(\vec N_{k,t}^{\textup{II-PA}}-1)\frac{k\vec N_{k,t}^{\textup{II-PA}}}{t}\nonumber\\
					& + \vec N_{k,t}^{\textup{II-PA}}\Big(1-\frac{(k-1)\vec N_{k-1,t}^{\textup{II-PA}}
					+k\vec N_{k,t}^{\textup{II-PA}}}{t}\Big)\Big]\nonumber\\
					= {} & \frac{(k-1)\bbE(\vec N_{k-1,t}^{\textup{II-PA}})}{t}+\Big(1-\frac{k}{t}\Big)
					\bbE(\vec N_{k,t}^{\textup{II-PA}}).
				\end{align}
				Now, in order to get the corresponding equation for what happens in $(m+1)$ steps, i.e., during the time interval
				from when there are $n$ vertices until when there are $(n+1)$ vertices, it is necessary to use (\ref{m1k}) and (\ref{mk})
				simultaneously.  In the same manner as we did for $k=1$, we take $t=(n+1)(m+1)-1$, so that at time $t+1$ the process will
				have exactly $(n+1)$ complete vertices.  Since $t$ is not multiple of $(m+1)$, we need to begin using (\ref{mk}) $(m-1)$ times,
				and then use (\ref{m1k}). Thus, iterating $m$ times, after some algebra  we obtain that for any $k\geq 2$,
				\begin{align}
					\label{mkgeneral}
					\bbE(\vec N_{k,t+1}^{\textup{II-PA}})= {} & \Big[\sum_{i=0}^{m-1}\frac{(k-1)\bbE
					(\vec N_{k-1,t-i}^{\textup{II-PA}})}{t-i}\prod_{j=0}^{i-1}\Big(1-\frac{k}{t-j}\Big)\Big] \notag \\
					&+\Big[\prod_{j=0}^{m-1}
					\Big(1-\frac{k}{t-j}\Big)\Big]\bbE(\vec N_{k,t-(m-1)}^{\textup{II-PA}}),
				\end{align}
				where the empty product (i.e., when $i=0$) is equal to unity.
				Let now $r=t-j$, and since $i\leq m$ and $m$ is fixed, then by Lemma \ref{lema1}  we have
				\begin{align}
					\label{product}
					\prod_{j=0}^{i-1}\Big(1-\frac{k}{t-j}\Big)=\prod_{r=t-(i-1)}^{t}\Big(1-\frac{k}{r}\Big)=1-\frac{ki}{t}+O\Big(\frac{k^2}{t^2}\Big).
				\end{align}
				Moreover,  observe that  $|\bbE(\vec N_{k-1,t-i}^{\textup{II-PA}})
				-\bbE(\vec N_{k-1,t-(m-1)}^{\textup{II-PA}})|\leq m+1-i$, since at each instant at most one edge is added. Thus
				\begin{align}
					\label{sumE}
					\sum_{i=0}^{m-1}\frac{\bbE(\vec N_{k-1,t-i}^{\textup{II-PA}})}{t-i}
					& = \sum_{i=0}^{m-1}\frac{\bbE(\vec N_{k-1,t-(m-1)}^{\textup{II-PA}})}{t}
					\Big(1+\frac{i}{t-i} \Big)+ O\Big(\frac{1}{t}\Big)\nonumber\\
					& = \frac{m\bbE(\vec N_{k-1,t-(m-1)}^{\textup{II-PA}})}{t}+O\Big(\frac{1}{t}\Big).
				\end{align}
				Then using  (\ref{product}) and (\ref{sumE}) and noting that $\frac{(k-1)\vec N_{k-1,t-i}^{\textup{II-PA}}}{t-i}\leq1$,
				we can write (\ref{mkgeneral}) as
				\begin{align}
					\bbE(\vec N_{k,t+1}^{\textup{II-PA}})
					= \frac{(k-1)m\bbE(\vec N_{k-1,t-(m-1)}^{\textup{II-PA}})}{t}
					+\Big(1-\frac{km}{t}\Big)\bbE(\vec N_{k,t-(m-1)}^{\textup{II-PA}})+O\Big(\frac{k}{t}\Big),
				\end{align}
				and using  the notation $\vec N(k,n)$  we obtain (\ref{nmk}).
			\end{proof}

			Theorem \ref{newman} gives the limit value to which
			$\vec N(k,n)/n$ converges when $n$ goes to infinity.
			However, before proving the limit, we need to argue that such  limit exists.
			\begin{lem}
				\label{lema3}
				Let $\vec N(k,n)$ be as in Lemma \ref{lema2}. Then, there exist  values $N_1(k)>0$ and $N_2(k)>0$
				such that 
				\begin{align*}
					\frac{\vec N(k,n)}{n}\longrightarrow N_1(k) \qquad \text{a.s.,}
				\end{align*}     
				and
				\begin{align*}
					\frac{\vec mk\vec N(k,n)}{(m+1)n}\longrightarrow N_2(k) \qquad \text{a.s.}
				\end{align*}
			\end{lem}
			\begin{proof}
				We make use of supermartingale's convergence theorem (see \cite {BillingsleyMeasure95}, Theorem 35.5)
				and equations (\ref{nm1k1}), (\ref{nmk1}), (\ref{nm1k}) and (\ref{nmk}).
				Consider first (\ref{nm1k1}) and (\ref{nmk1}) and observe that since $\vec N(1,n)/((n+1)(m+1)-1)\leq 1$,     
				\begin{align}
					\label{2831}
					\bbE\vec N(1,n+1)\leq \bbE\vec N(1,n)+1, 
				\end{align}
				while for
				(\ref{nm1k}) and (\ref{nmk}),   
				\begin{align}
					\label{3440}
					\bbE\vec N(k,n+1)\leq \bbE\vec N(k,n)+1+O\Big(\frac{k}{n}\Big). 
				\end{align}
				Let $\mathcal{H}_n$ be the filtration generated by the process $\{\vec N(k,n),\vec N(k-1,n)\}_n$ until time
				$n$, i.e., $\mathcal{H}_n:=\sigma(N(k,j),N(k-1,j);\: 0\leq j\leq n )$.
				If $k=1$, let $Z(1,n)=(\vec N(1,n)-n)/n$,  then by  (\ref{2831}),  
				\begin{align}
					\bbE [Z(1,n+1)\mid \mathcal{H}_{n}]\leq \frac{\vec N(1,n)+1-(n+1)}{n+1}\leq \frac{\vec N(1,n)+1-(n+1)}{n}= Z(1,n),
				\end{align}
				as $N(1,n)$ is $\mathcal{H}_n$-measurable. Hence,  $\{Z(1,n)\}_n$ is a supermartingale and in order to apply
				supermartingale convergence theorem to $\{Z(1,n)\}_n$, it remains to prove that
				\begin{align*}
					\sup_n\bbE(|Z(1,n)|)<\infty.
				\end{align*}
				This is true as
				\begin{align}
					\bbE(|Z(1,n)|)=\bbE\big[\bbE\bigl(|Z(1,n)|\bigr| \mathcal{H}_{n-1}\bigr)\big]\leq\frac{1}{n}(\bbE\vec N(1,n-1)+1+n)<\infty,
				\end{align}
				having used that $\vec N(k,n)/n\leq1$, for any $n\geq1$. 

				When $k\geq2$, i.e.,  for (\ref{nm1k}) and (\ref{nmk}), note first  that  if $f(n)=O(k/n)$, then there exists
				$M>0$ such that  $|f(n)|=Mk/n+|\delta|$,  where $|\delta|<k/n$. Since $k/n\leq 1$, then there exists $M$ such that
				$|f(n)|\leq M+1$.  Thus,  take $Z(k,n)=[\vec N(k,n)-n(c+1)]/n$, with $c=M+1$,  then by  (\ref{2831}) we also get
				that $\{Z(k,n)\}_n$ is a supermartingale, and similarly as we did above we also show that $\sup_n\bbE(|Z(k,n)|)<\infty$.
				In this manner we have proved that  $Z(k,n)$  converges almost surely,
				thus $\vec N(k,n)/n$ converges almost surely.  In perfect analogy we can
				prove that   $mkZ(k,n)/(m+1)$ converges almost surely, and thus obtain that $mk\vec N(k,n)/n(m+1)$ also converges
				almost surely.
			\end{proof}

			In order to determine such a limit, we still need to prove the following lemma.
			\begin{lem}
				\label{lema4}
				Let $p_k:=\lim_{n\longrightarrow\infty}\bbE \vec N(k,n)/n$. Then,
				\begin{align}
					\label{pkk}
					p_k=\frac{(1+1/m)\Gamma(k)\Gamma(2+1/m)}{\Gamma(k + 2 + 1/m)}, \qquad k \ge 1.
				\end{align}
			\end{lem}
			\begin{proof}
				Observe that for a function $f(k)$, 
				\begin{align}
					\label{taniguchi}
					\frac{mf(k)}{(n+1)(m+1)-1}=\frac{mf(k)}{n(m+1)}\Big(1-\frac{m}{(n+1)(m+1)-1}\Big)=\frac{mf(k)}{n(m+1)}
					+O\Big(\frac{f(k)}{n^2}\Big).
				\end{align}
				By using \eqref{taniguchi} we can write the equations (\ref{nm1k1}), (\ref{nmk1}), (\ref{nm1k}) and (\ref{nmk})
				as follows.
				For $m=1$, $k=1$,
				\begin{align}
					\label{1}
					\bbE\vec N(1,n+1)= 1+\Big(1-\frac{1}{n(m+1)}\Big)\bbE\vec N(1,n)+O\Big(\frac{1}{n}\Big),
				\end{align}
				for $m>1$, $k=1$,
				\begin{align}
					\label{3}
					\bbE\vec N(1,n+1)= 1+\Big(1-\frac{m}{n(m+1)}\Big)\bbE\vec N(1,n)+O\Big(\frac{1}{n}\Big),
				\end{align}
				for $m=1$, $k\geq 2$,
				\begin{align}
					\label{2}
					\bbE\vec N(k,n+1)= \frac{(k-1)\bbE\vec N(k-1,n)}{n(m+1)}+\Big(1-\frac{k}{n(m+1)}\Big)\bbE\vec N(k,n)+O\Big(\frac{1}{n}\Big),
				\end{align}
				and for $m>1$, $k\geq 2$,
				\begin{align}
					\label{4}
					\bbE\vec N(k,n+1)= \frac{m(k-1)\bbE\vec N(k-1,n)}{n(m+1)}+\Big(1-\frac{mk}{n(m+1)}\Big)\bbE\vec N(k,n)+O\Big(\frac{k}{n}\Big).
				\end{align}
				Looking at (\ref{1}), (\ref{2}), (\ref{3}) and (\ref{4}), we remark that they can be unified as
				\begin{align}
					\label{g5}
					\bbE\vec N(k,n+1)= g(k-1,n)+\Big(1-\frac{b}{n}\Big)\bbE\vec N(k,n)+ \mathscr{E}_n,
				\end{align}
				where $b=km/(m+1)$,  $g(0,n)=1$,  $g(k-1,n)=(mk/n(m+1))\bbE\vec N(k,n)$ for $k\geq2$,  and $\mathscr{E}_n= O(1/n)$
				if $m=1$ and of order  $O(k/n)$  if $m>1$. We underline that $k$ could be a function of $n$ and hence in general
				$O(k/n)$ can be different of $O(1/n)$.

				Note now that when the first complete vertex appears, it has in-degree equal to  $(m+1)$, so $\vec N(k,1)=0$
				for any $k\neq (m+1)$, and $\vec N(m+1,1)=1$.  Iterating (\ref{g5}) we have,  if $k\neq (m+1)$, 
				\begin{align}
					\label{gk}
					\bbE\vec N(k,n+1)= \sum_{i=0}^{n-1}g(k-1,n-i)\prod_{j=0}^{i-1}\Big(1-\frac{b}{n-j}\Big)+\sum_{i=0}^{n-1} \mathscr{E}_{n-i},
				\end{align}
				while, if $k=m+1$,
				\begin{align}
					\label{gm}
					\bbE\vec N(k,n+1)= \sum_{i=0}^{n-1}g(k-1,n-i)\prod_{j=0}^{i-1}\Big(1-\frac{b}{n-j}\Big)+\prod_{j=0}^{n-1}
					\Big(1-\frac{b}{n-j}\Big)+\sum_{i=0}^{n-1} \mathscr{E}_{n-i}.
				\end{align} 
				To solve  (\ref{gm}), 
				let $s=n-i$ and $r=n-j$ so that
				\begin{align*}
					\prod_{j=0}^{i-1}\Big(1-\frac{b}{n-j}\Big) = \prod_{r=s+1}^{n}\Big(1-\frac{b}{r}\Big),
				\end{align*}
				then  observe that if  $s< \lfloor{b}\rfloor$,
				\begin{align*}
					\prod_{r=s+1}^{n}\Big(1-\frac{b}{r}\Big)
					= \prod_{r=s+1}^ {\lfloor{b}\rfloor}\Big(1-\frac{b}{r}\Big)\prod_{r=\lfloor{b}\rfloor+1}^{n}\Big(1-\frac{b}{r}\Big),
				\end{align*}
				which is equal either to 0 if  $b=\lfloor{b}\rfloor$ or to
				\begin{align*}
					(-1)^{\lfloor{b}\rfloor}\prod_{i=1}^{\lfloor{b}\rfloor}
					\frac{(b-i)}{(\lfloor{b}\rfloor-i+1)}\prod_{r=\lfloor{b}\rfloor+1}^{n}\Big(1-\frac{b}{r}\Big),
				\end{align*}
				if  $b\neq \lfloor{b}\rfloor$.
				Applying  Lemma \ref{lema1}  (note that to apply this lemma is necessary to have $b/r<1$, i.e., $r\geq \lfloor{b}\rfloor+1 $)
				we have
				\begin{align*}
					\label{prod2}
					\prod_{r=s+1}^{n}\Big(1-\frac{b}{r}\Big)=
					\begin{cases}
						0, & s<\lfloor{b}\rfloor, \: b=\lfloor{b}\rfloor, \\
						O\big(\frac{\lfloor{b}\rfloor}{n}\big)^{\lfloor{b}\rfloor}, & s<\lfloor{b}\rfloor, \: b\neq\lfloor{b}\rfloor,\\
						\big(1+O\big(\frac{n-s}{sn}\big)\big)\big(\frac{s}{n}\big)^b, & s\geq\lfloor{b}\rfloor.
					\end{cases}
				\end{align*}
				Using this  and (\ref{1/r}), formula (\ref{gm}) can be written as
				\begin{align}
					\bbE\vec N(k,n+1) & = \sum_{s=\lfloor{b}\rfloor}^{n}g(k-1,s)
					\Big(\frac{s}{n}\Big)^b\Big(1+O\Big(\frac{n-s}{sn}\Big)\Big)
					+O\Big(\frac{\lfloor{b}\rfloor}{n}\Big)^{\lfloor{b}\rfloor}+\mathscr{E}\nonumber\\
					& = \sum_{s=\lfloor{b}\rfloor}^{n}g(k-1,s)\Big(\frac{s}{n}\Big)^b+\mathscr{E},
				\end{align}
				where the error term $\mathscr{E}$ is of order  $O(\ln n)$ if $m=1$ and  of order  $O(k\ln n)$  if $m>1$.
				It is not difficult to see that, following a similar procedure,  we can  get the same solution for  (\ref{gk}).

				Now, by Lemma \ref{lema3} we know  that there exist some $N_1(k)>0$ and some $N_2(k)>0$,
				such that $\vec N(k,n)/n\longrightarrow N_1(k)$ and $mk\vec N(k,n)/[n(m+1)]\longrightarrow N_2(k)$ almost surely
				(observe  that in order to guarantee $a.s.$ convergence, we will need to take $k$ independent of $n$,
				hence we will obtain $N_1(k)$ and $N_2(k)$ strictly positive).
				Thus,  by the dominated convergence theorem
				$p_k:=\lim_{n\longrightarrow} \bbE\vec N(k,n)/n$,
				and  for $k\geq 2$, $g(k-1):=\lim_{n\longrightarrow}g(k-1,n)$, exist.
				
				Note that $g(k-1)=m(k-1)p_{k-1}/(m+1)$, and 
				let us write $g(k-1,n)=g(k-1)+O(\varepsilon_n)$, where $\varepsilon_n\longrightarrow 0$ as $n\longrightarrow \infty$.  Hence,
				\begin{align}
					\sum_{s=\lfloor{b}\rfloor}^{n}g(k-1,s)\Big(\frac{s}{n}\Big)^b=\frac{g(k-1)}{n^b} \sum_{s=\lfloor{b}\rfloor}^{n}
					s^b+\frac{1}{n^b} \sum_{s=\lfloor{b}\rfloor}^{n}O(\varepsilon_s) s^b,
				\end{align}
				and using that $\sum_{s=\lfloor{b}\rfloor}^{n} s^b=n^{b+1}/(b+1)+ o(n^{b+1})$ (see 3.II of \cite{IAAlg}),
				we obtain
				\begin{align}
					\label{limitE}
					\frac{\bbE\vec N(k,n+1)}{n+1}=
					\begin{cases}
						\frac{g(k-1)}{b+1}+O\left(\frac{\ln n}{n}\right), & m=1,\\
						\frac{g(k-1)}{b+1}+O\left(\frac{k\ln n}{n}\right), & m>1.
					\end{cases}
				\end{align}
				Observe that when $m>1$ we would need more restrictions on $k$ in order to determine the limit of $\bbE\vec N(k,n+1)/(n+1)$.
				Indeed it should satisfy that $k\ln n/n\longrightarrow 0$ as $n\longrightarrow\infty$, but that is
				true since we are taking $k$ fixed, i.e., independent of $n$.  Thus, by (\ref{limitE}),
				\begin{align}
					\label{limitEE}
					p_k=\lim_{n\longrightarrow\infty}\frac{\bbE\vec N(k,n+1)}{n+1}=
					\begin{cases}
						\frac{m+1}{2m+1}, & k=1,\\
						\frac{m(k-1)p_{k-1}}{m(k+1)+1}, & k>1.
					\end{cases}
				\end{align}
				Solving (\ref{limitEE}) recursively  we get (\ref{pkk}).
			\end{proof}

			\begin{proof}[Proof of Theorem \ref{newman}]
				We follow the approach of Dorogovtsev, Mendes, and Samukhin \cite{PhysRevLett.85.4633}, that uses master equations for the expected
				value of the number of vertices with in-degree $k$. To obtain the exact equations we need to consider two stages.
				For the first one we consider what happens in one step of the process, during which the number of vertices of in-degree $k$
				can be increased by counting also some vertices coming from those having previously in-degree $(k-1)$ or in-degree $(k+1)$,
				and then we consider what happens in $(m+1)$ steps, thus obtaining the change of the vertices in-degree in an interval of
				time starting when the process has $n$ vertices, until it has $(n+1)$ vertices.  This part corresponds to finding
				the equations (\ref{nm1k1}), (\ref{nmk1}), (\ref{nm1k}) and (\ref{nmk}) given in Lemma \ref {lema2}.
				For the second stage, we iterate the previous equations with respect to $n$ and obtain the limit of
				$\bbE\vec N(k,n)/n$ as $n\longrightarrow\infty$. This part was proved in Lemma \ref {lema4} determining
				(\ref{pkk}).  

				Finally, we use Azuma--Hoeffding inequality (\ref{AzHoff}) to obtain (\ref{IIPam}).   Let $\mathcal{F}_t$ be the natural
				filtration generated by the process $\{\vec N_{k,t}^{\textup{II-PA}}\}$  up to time $t$. Then, in the same way as it was
				explained for to Simon model, Section \ref{Simonmodel}, it is easy to show that for $s\leq t$,
				$Z_s^{\textup{II-PA}} = \bbE(\vec N_{k,t}^{\textup{II-PA}}|\mathcal{F}_s)$ is a martingale such that,
				$|Z_s^{\textup{II-PA}} -Z_{s-1}^{\textup{II-PA}}| \leq 1$,  $Z_t^{\textup{II-PA}}=\vec N_{k,t}^{\textup{II-PA}}$
				and $Z_0^{\textup{II-PA}}=\bbE\vec N_{k,t}^{\textup{II-PA}}$ (at time $t=0$ the random graph is the empty graph).  
				Thus  by (\ref{AzHoff}) we get that for every  $\epsilon_n \gg 1/\sqrt{n}$, e.g.\ take $\epsilon_n=\sqrt{\ln n/n}$,
				\begin{align*}
					\bbP\Big(\Big|\frac{\vec N_{k,t}^{\textup{II-PA}}}{n}-\frac{\bbE\vec N_{k,t}^{\textup{II-PA}}}{n}\Big|\geq\epsilon_n\Big)
					\leq \exp\Big(-\frac{(n \epsilon_n)^2}{2t}\Big)\longrightarrow0,
				\end{align*}
				as $n$ goes to infinity. Here $t=n(m+1)+i,$ for $i=0,1,\dots,m$. Thus we obtain that for $t=n(m+1)$,
				\begin{align*}
					\frac{\vec N_{k,t}^{\textup{II-PA}}}{n}\longrightarrow \frac{(1+1/m)\Gamma(k)\Gamma(2+1/m)}{\Gamma(k + 2 + 1/m)},
				\end{align*}
				in probability as $n\longrightarrow\infty$. However, by Lemma \ref {lema3} we actually have an almost sure convergence.
			\end{proof}

		\subsubsection{The Price model}

			\label{PriceApprox}
			Let  $M_1,M_2,\dots$ be  independent and identically distributed random variables with $\bbE(M_i)=m$,
			where $m$ is a positive rational number and $\bbV(M_i)=\sigma^2$. Furthermore, let $(\tilde{G}_m^n)_{n\geq 1}$ be the random
			graph process defining the Price model as in Section \ref{Price}  and take $k_0=1$.
			If $\vec N_{k,n}^{\textup{Price}}$ denotes the number of vertices with in-degree equal
			to  $k$  in $(\tilde{G}_m^n)_{n\geq 1}$, $k\geq1$, then
			\begin{equation}
				\label{PricePrice}
				\frac{\vec N_{k,n}^{\textup{Price}}}{n}\longrightarrow
				\frac{(1+1/m)\Gamma(k)\Gamma(2+1/m)}{\Gamma(k + 2 + 1/m)}
			\end{equation}
			almost surely as $n\longrightarrow\infty$.
				
			A rigorous analysis of the previous result can be made using Chebyschev's inequality and following the same lines as
			in the proof of Theorem \ref{newman} for the II-PA model.  Hence, we limit ourselves to present a scheme of the proof. 

			\begin{enumerate}
						
				\item In the mathematical description of the Price model, we saw that $\tilde{G}_m^{n+1}$ is formed from $\tilde{G}_m^n$ by
					adding a new vertex $v_{n+1}$ with $k_0$ directed loops, and from it a random number, $M_{n+1}$, of directed edges to
					different old vertices. This happens with probabilities proportional to their in-degrees
					as in (\ref{Priceprob}). Conditioning on the number of vertices with in-degree $k$ when there are $n$ vertices,
					we obtain
					\begin{align}
						\label{PriceEQ1}
						\bbE \vec N_{k_0,n+1}^{\textup{Price}}&= \bbE \left[\bbE(\vec N_{k_0,n+1}^{\textup{Price}}
						\mid\vec N_{k_0,n}^{\textup{Price}}) \right]\nonumber\\
						&=\bbE \left[1+(\vec N_{k_0,n}^{\textup{Price}}-1)\frac{M_{n+1}k_0\vec
						N_{k_0,n}^{\textup{Price}}}{nk_0+\sum_{i=1}^n M_{i}}
						+\vec N_{k_0,n}^{\textup{Price}}\left(1-\frac{M_{n+1}k_0\vec
						N_{k_0,n}^{\textup{Price}}}{nk_0+\sum_{i=1}^n M_{i}}\right)\right]\nonumber\\
						&= 1+\bbE\left[\left(1-\frac{M_{n+1}k_0}{nk_0+\sum_{i=1}^n M_{i}} \right)\vec N_{k_0,n}^{\textup{Price}}\right],
					\end{align}
					and, for $k>k_0$,
					\begin{align}
						\label{PriceEQ2} 
						\bbE \vec N_{k,n+1}^{\textup{Price}}
						= {} &  \bbE \left[(\vec N_{k,n}^{\textup{Price}}+1)\frac{M_{n+1}(k-1)
						\vec N_{k-1,n}^{\textup{Price}}}{nk_0+\sum_{i=1}^n M_{i}}
						+(\vec N_{k,n}^{\textup{Price}}-1)\frac{M_{n+1}k\vec N_{k,n}^{\textup{Price}}}{nk_0+\sum_{i=1}^n m_{i}}\right.\nonumber\\
						& + \left. \vec N_{k,n}^{\textup{Price}}\left(1-\frac{M_{n+1}(k-1)
						\vec N_{k-1,n}^{\textup{Price}}}{nk_0+\sum_{i=1}^n M_{i}}-\frac{M_{n+1}
						k\vec N_{k,n}^{\textup{Price}}}{nk_0+\sum_{i=1}^n M_{i}}\right)\right]\nonumber\\
						= {} & \bbE\left[ \frac{M_{n+1}(k-1)\vec N_{k-1,n}^{\textup{Price}}}{nk_0+\sum_{i=1}^nM_{i}}+\left(1-\frac{M_{n+1}k}{nk_0
						+\sum_{i=1}^n M_{i}}\right)\vec N_{k,n}^{\textup{Price}}\right].
					\end{align}
				
				\item Take $k_0=1$, $Y=n+\sum_{i=1}^n M_{i}$ and $\epsilon_n=\sqrt{\ln n/n}$. By Chebyschev's inequality,
					\begin{align}
						\label{ChebApprox}
						\bbP[|Y-n(1+m)|> n\epsilon_n]\leq\frac{\sigma^2}{\sqrt{n}\ln n}.
					\end{align}						
					Let $X$ be another random variable such that $\bbE(X/Y)$ is bounded, and $0\leq\bbE(X)\leq (k-1)mn$.
					In addition, define the event $E_n:=\{n(1+m-\epsilon_n)\leq Y\leq n(1+m+\epsilon_n)\}$.
					Conditioning on $E_n$ and applying (\ref{ChebApprox}),
					\begin{align}
						\label{EXY}
						\bbE(X/Y)\approx\bbE(X/Y\mid Y\in E_n)+ O\left(\frac{1}{\sqrt{n}\ln n} \right),
					\end{align}
					because $\bbE(X/Y)$ is bounded.
					Furthermore, note that
					\begin{align*}
						\frac{\bbE(X)}{n(1+m)}\left(1-\frac{\epsilon_n}{1+m+\epsilon_n}\right)\leq
						\bbE\left(\left. \frac{X}{Y} \right| Y\in E_n\right)
						\leq \frac{\bbE(X)}{n(1+m)}\left(1-\frac{\epsilon_n}{1+m-\epsilon_n}\right),
					\end{align*}
					thus $\bbE\left(X/Y\mid Y\in E_n\right)=\frac{\bbE(X)}{n(1+m)}+O(\sqrt{\ln n/n})$. Replacing this in (\ref{EXY}) we have
					\begin{align}
						\label{EXYapprox}
						\bbE(X/Y)\approx \frac{\bbE(X)}{n(1+m)}+O(\sqrt{\ln n/n}).
					\end{align}	
					Using (\ref{EXYapprox}) in (\ref{PriceEQ1}) with $X=M_{n+1}\vec N_{1,n}^{\textup{Price}}$
					and in (\ref{PriceEQ2}) with $X=M_{n+1}(k-1)\vec N_{k-1,n}^{\textup{Price}}$
					and $X=M_{n+1}k\vec N_{k,n}^{\textup{Price}}$, respectively, we get 
					from (\ref{PriceEQ1}) and (\ref{PriceEQ2}) that
					\begin{align}
						\label{PriceEQ11}
						\bbE \vec N_{k_0,n+1}^{\textup{Price}}&\approx
						1+\left(1-\frac{m}{n(1+m)} \right)\bbE\vec N_{k_0,n}^{\textup{Price}}+O(\sqrt{\ln n/n}),
					\end{align}
					and, for $k>1$,
					\begin{align}
						\label{PriceEQ22} 
						\bbE \vec N_{k,n+1}^{\textup{Price}}
						&\approx \frac{m(k-1)\bbE\vec N_{k-1,n}^{\textup{Price}}}{n(1+m)}
						+\left(1-\frac{mk}{n(1+m)}\right)\bbE\vec N_{k,n}^{\textup{Price}}+O(\sqrt{\ln n/n}).
					\end{align}
							
				\item Note now that  (\ref{PriceEQ11}) and (\ref{PriceEQ22}) are almost the same as	(\ref{nmk1})
					and (\ref{nmk}) for the II-PA model, respectively. In order to derive (\ref{PricePrice}) we then proceed
					as in the proof of Theorem \ref{newman}.  More specifically, to ensure the existence of the limit value of
					$\vec N_{k,n}^{\textup{Price}}/n$ as $n\rightarrow\infty$, we use supermartingale's convergence theorem
					(see \cite {BillingsleyMeasure95}, Theorem 35.5), in analogy to Lemma \ref{lema3}.
					Then we find that 
					\begin{align}
						\label{pkkPrice}
						\lim_{n\rightarrow\infty}\frac{\bbE \vec N_{k,n}^{\textup{Price}}}{n}=\frac{(1+1/m)\Gamma(k)
						\Gamma(2+1/m)}{\Gamma(k + 2 + 1/m)},
						\qquad k \ge 1,
					\end{align}
					as in Lemma \ref{lema4}. Finally by  Azuma--Hoeffding inequality (\ref{AzHoff}) we obtain 
					\begin{align*}
						\frac{\vec N_{k,n}^{\textup{Price}}}{n}\rightarrow\frac{(1+1/m)\Gamma(k)\Gamma(2+1/m)}{\Gamma(k + 2 + 1/m)}, \qquad k \ge 1,
					\end{align*} 
					in probability as $n\rightarrow\infty$. By Lemma \ref{lema3} the result follows almost surely.
				\end{enumerate}

				Notice that the Price model
				is by definition equivalent to the II-PA if $M_i=m=1$ almost surely. Moreover, Price and II-PA models have the same limit
				in-degree distribution when $\bbE(M_i)=m$.
						
		\subsection{Relation between Simon and Yule models}
		
			\label{Relation2}
			Bearing in mind the construction of Yule model as explained in Section \ref{YuleDef}, we underline that the inter-event times
			of in-links appearance and those related to creation of new vertices are exponentially distributed. In order to relate
			Yule and Simon models we investigate here the inter-event times characterizing Simon model showing that a suitable rescaling
			in the limit leads to exponential random variables. 
			The idea is to identify two different processes which conditionally describe Simon model, and clarify
			how these are related with the two Yule processes which define a Yule model.   

			The next theorem together with Remarks \ref{FirstYule} and \ref{SecondYule} allows us to recognize the first process inside
			a Simon model behaving asymptotically  as a Yule process with parameter $(1-\alpha)$, while Theorem \ref{S-Yule2} and
			Remark \ref{ObsS-Yule2} determine the second process which behaves  asymptotically  as a Yule process with parameter equal to one.
			The first process models how the vertices get new in-links, thus at each moment a new vertex appears, a process starts.
			On the other hand, the second process is related to how the vertices appear.

			Let  $(G_{\alpha}^t)_{t\geq1}$ be the random graph process associated to Simon model of parameter $\alpha$, as  described in
			Subsection \ref{Simonmodel}, and let  $\{ \vec d(v_i,t)\}_{t\geq t^i_0}$ be the in-degree process associated to the vertex $v_i$,
			which appears at time $t_0^i$, i.e., $t^i_0=\min \{t \colon \vec d(v_i,t)=1\}$ (note that $\vec d(v_i,t^i_0)=1$ as the vertex
			appears together with a directed loop in the model).  

			Our first focus is on the study of the distributions of  the waiting times between the instant in which each vertex has in-degree $k$,
			till that in which it has in-degree $k+1$. Formally, we study  the distribution of the random variables
			$W_k^i=t^i_k-t^i_{k-1}$, $k \ge 1$, where $t^i_j=\min \{t: \vec d(v_i,t)=j+1\}$ for $j=0,1,2,\dots$\,. 

			\begin{teo}
				\label{S-Yule}
				Let  $z=\ln\Big(1+\frac{x}{t^i_{k-1}-1}\Big)$, $k \ge 1$, $x > 0$. It holds
				\begin{equation}
					\Big|\bbP(W_k^i\leq x)-\bbP(Z_k^i\leq z)\Big|<O\Big(\frac{1}{t^i_{k-1}}\Big),
				\end{equation}
				where $Z_k^i$ is an exponential random variable of parameter $(1-\alpha) k$. 
			\end{teo}

			\begin{obs}
				\label{FirstYule}
				Theorem \ref{S-Yule} states  that
				for any $t^*$ large enough but fixed, 
				\begin{equation}
					\Big|\bbP(W_k^j\leq x)-\bbP(Z_k^j\leq z)\Big|<O\Big(\frac{1}{t^*}\Big),
				\end{equation}
				$\forall j\geq \min\{i \colon t^i_0\geq t^*\} $, and for $k\geq1$. This means that from a fixed but large time $t^*$, all the
				waiting times $W_k^j$ are approximately exponential random variables, with an error term smaller than
				$O\left(1/t^*\right)$. 
			\end{obs}

			\begin{proof}[Proof of Theorem \ref{S-Yule}]
				By the preferential attachment probabilities (\ref{parule}) of $(G_{\alpha}^t)_{t\geq1}$, for $x\geq1$, we have
				\begin{align}
					\label{Wx}
					\bbP[W_k^i=x]&=\Big(\frac{\vec d(v_i,t^i_{k-1})(1-\alpha)}{t^i_{k-1}+x-1}\Big)
					\Big(1-\frac{\vec d(v_i,t^i_{k-1})(1-\alpha)}{t^i_{k-1}+x-2}\Big)
					\dots\Big(1-\frac{\vec d(v_i,t^i_{k-1})(1-\alpha)}{t^i_{k-1}}\Big)\nonumber\\
					&=\Big(\frac{k(1-\alpha)}{t^i_{k-1}+x-1}\Big)\Big(1-\frac{k(1-\alpha)}{t^i_{k-1}+x-2}\Big)\dots
					\Big(1-\frac{k(1-\alpha)}{t^i_{k-1}}\Big)\nonumber\\
					&=\Big(\frac{k(1-\alpha)}{t^i_{k-1}+x-1}\Big)\prod_{r=t^i_{k-1}}^{t^i_{k-1}+x-2}\Big(1-\frac{k(1-\alpha)}{r}\Big),
				\end{align}
				as $t^i_{k-1}>k$. Then $k(1-\alpha)/r<1$, so we can apply Lemma \ref{lema1} to the product to obtain 
				\begin{align}
					\label{prod}
					\prod_{r=t^i_{k-1}}^{t^i_{k-1}+x-2}\Big(1-\frac{k(1-\alpha)}{r}\Big)
					=\Big(\frac{t^i_{k-1}-1}{t^i_{k-1}+x-1}\Big)^{k(1-\alpha)} \Big(1+O\Big(\frac{x}{(t^i_{k-1}+x-2)(t^i_{k-1}-1)}\Big)\Big).
				\end{align}
				Thus, using (\ref{Wx}), (\ref{prod}), the Euler--Maclaurin formula,
				\begin{align*}
					\sum_{j=1}^{n}\frac{1}{j^s}=\frac{1}{n^{s-1}}-s\int_{1}^{n}\frac{\lfloor{y}\rfloor}{y^{s+1}}dy,
				\end{align*}
				with $s\in\bbR \setminus \{1\}$ (see \cite{IAAlg}) and the fact that $\lfloor{y}\rfloor\leq y$, we arrive at
				\begin{align}
					\bbP[W_k^i\leq x] = {} & \sum_{w=1}^{x} \Big(\frac{k(1-\alpha)}{t^i_{k-1}+w-1}\Big)
					\Big(\frac{t^i_{k-1}-1}{t^i_{k-1}+w-1}\Big)^{k(1-\alpha)}\Big(1+O\Big(\frac{w}{(t^i_{k-1})^2+wt^i_{k-1}}\Big)\Big)\nonumber\\
					= {} & k(1-\alpha)(t^i_{k-1}-1)^{k(1-\alpha)}\sum_{w=1}^{x}\Big(\frac{1}{t^i_{k-1}+w-1}\Big)^{k(1-\alpha)+1}
					\Big(1+O\Big(\frac{w}{(t^i_{k-1})^2+wt^i_{k-1}}\Big)\Big)\nonumber\\
					< {} & \Big(1+O\Big(\frac{1}{t^i_{k-1}}\Big)\Big)k(1-\alpha)(t^i_{k-1}-1)^{k(1-\alpha)}
					\sum_{j=t^i_{k-1}}^{t^i_{k-1}+x-1}\Big(\frac{1}{j}\Big)^{k(1-\alpha)+1}\nonumber\\
					= {} & \Big(1+O\Big(\frac{1}{t^i_{k-1}}\Big)\Big)k(1-\alpha)(t^i_{k-1}-1)^{k(1-\alpha)}\nonumber\\
					& \times \Big(\frac{1}{(t^i_{k-1}+x-1)^{k(1-\alpha)}}-
					\frac{1}{(t^i_{k-1})^{k(1-\alpha)}}+(k(1-\alpha)+1)
					\int_{t^i_{k-1}}^{t^i_{k-1}+x-1}\frac{\lfloor{y}\rfloor}{y^{k(1-\alpha)+2}}dy\Big)\nonumber\\
					< {} & \Big(1+O\Big(\frac{1}{t^i_{k-1}}\Big)\Big)(t^i_{k-1}-1)^{k(1-\alpha)}
					\Big[\frac{1}{(t^i_{k-1}-1)^{k(1-\alpha)}}-\frac{1}{(t^i_{k-1}+x-1)^{k(1-\alpha)}}\Big]\nonumber\\
					= {} & \Big(1+O\Big(\frac{1}{t^i_{k-1}}\Big)\Big)\Big[1-\exp\Big[-k(1-\alpha)\ln\Big(1+\frac{x}{t^i_{k-1}-1}\Big)\Big]\Big].
				\end{align}
				Thus, we get
				\begin{align}
					\Big|\bbP(W_k^i\leq x)-\left[1-\exp\big(-k(1-\alpha)\ln(1+x/(t^i_{k-1}-1))\big)\right]
					\Big|<O\Big(\frac{1}{t^i_{k-1}}\Big).\nonumber
				\end{align}
				Then, by taking $z=\ln(1+x/(t^i_{k-1}-1))$, it holds 
				\begin{align}
					\Big|\bbP(W_k^i\leq x)-\bbP(Z_k^i\leq z)\Big|<O\Big(\frac{1}{t^i_{k-1}}\Big),\nonumber
				\end{align}  
				where $Z_k^i$ is a  random variable exponentially distributed with parameter  $(1-\alpha)k$.
			\end{proof}

			\begin{obs}
				\label{SecondYule}
				Note that  the knowledge of $\{W_k^i\}, k\geq 1$, is equivalent to the knowledge of
				$\vec d(v_i,t)$, as $\vec d(v_i,t):=\min\{k \colon\sum_{b=1}^{k}W^i_b> (t-t^i_0)\}$. Thus, due to Theorem \ref{S-Yule},
				the process $\{ \vec d(v_j,t)\}_{t\geq t^j_0}$, $\forall j\geq \min\{i \colon t^i_0\geq t_0\} $, behaves asymptotically
				as a Yule process with parameter $(1-\alpha)$. 
			\end{obs}
 
			Let us now consider the growth of the vertices in Simon model, where at each instant of time $t$, a new vertex is created
			with a fixed probability $\alpha$. 
			This fact can be re-thought from  a different perspective as follows.  
			Remember that in Simon model the  number of vertices at time $t$ is a random variable $V(t)$, distributed Binomially,
			$\text{Bin}(t,\alpha)$, and that at each instant of time, one and only one vertex can appear. 
			Think for a moment that we know the number of vertices at time $t$, 
			then, conditionally on that,  at time $t+1$ choose uniformly at random an existing vertex, i.e., with probability $1/V(t)$
			select one vertex, and  with probability $\alpha$ duplicate it. Note that, as time increases, each existing vertex  may give
			birth to a new vertex with probability $\alpha/V(t)$.  In this way we have that  a new vertex appears with constant probability
			$\alpha$; since there are $V(t)$ vertices, then the probability of the birth of a new vertex is $V(t)(\alpha /V(t)) = \alpha$.

			Now  fix a time, take for example  $t_0^i$, the time when the $i$th vertex appears, so $V(t_0^i)=i$.
			For each of the  existing vertices at time $t_0^i$, say $v_j$, $1\leq j\leq i$,  define the birth process
			$\{D_j(t)\}_{t\geq t_0^i}$  of all its descendants  as follows.
			Start at time $t_0^i$ with one vertex, $v_j$.  Since  at time $t+1$ each existing vertex in Simon model  may give birth
			to a new vertex with probability $\alpha/V(t)$,  then if at time $t$  the number of vertices descendent
			of $v_j$ (i.e., itself $+$ its children $+$ its grandchildren $+$ etc.) is $k$, the probability that a new descendent
			of $v_j$ appears  at  time $t+1$ is $k\alpha/V(t)$.
			Formally, let  $D_j(t)$ be the total number  of descendent of $v_j$ at time $t$ with  $D_j(t_0^i)=1$ (itself),
			then if $D_j(t)=k$, $k\geq 1$, the probability that a new descendent of $v_j$ appears  at  time $t+1$ is $k\alpha/V(t)$.

			Observe that since at each time we are selecting one and only one vertex in Simon model to duplicate,
			the probability of either no duplications or more than two at each instant of time $t$ is zero. Clearly this is different from the case
			in which we had taken independent processes $\{D_j(t)\}_{t\geq t_0^i}$, therefore,  they  are dependent.
			However, by  definition,  these processes,  are equal in  distribution, i.e.,
			$\bbP(D_j(t)\leq d)$ is the same for each $1\leq j\leq i$.

			We will see in the following theorem that the  processes $\{D_j(t)\}_{t\geq t_0^i}$, $1\leq j\leq i$, converge in distribution
			to  Yule processes with parameter $1$, i.e.,  if at time $t$,  $D_j(t)=k$, $k\geq 1$, the number of steps up to see the next
			descendent of $v_j$,  converges in distribution to an exponential random variable with parameter $k$.
			Thus, starting with $i$ vertices,  we will see that from $t_0^i$, the process of appearance  of new vertices in Simon model
			approximates  $i$ dependent but identically distributed  Yule processes with parameter $1$. If the interest is to study
			the asymptotic characteristics of a uniformly chosen  random vertex in Simon model,  we could do that first by choosing
			uniformly at random a Yule process with parameter $1$, and then, by choosing uniformly at random an individual belonging to it.

			Formally, let  $(G_{\alpha}^t)_{t\geq1}$ be the random graph process corresponding to Simon model (described in  Subsection
			\ref{Simonmodel}) and, as above, let $t_0^i$ be the time when the $i$th vertex appears. Then,  for each  vertex in this process
			up to time $t_0^i$, say $v_j$,   $1\leq j\leq i$,
			let $\mathcal{Y}_k^j$  be the random variables  $\mathcal{Y}_k^j:=\ell^j_k-\ell^j_{k-1}$, for $k=1,2,\dots$,
			where  $\ell^j_0=t_0^i$,  and  $\ell^j_k$ is the minimum $t$ when there are exactly $k+1$ descendants of $v_j$ in
			$\{D_j(t)\}_{t\geq t_0^i}$, $k\geq 1$. Hence  $\mathcal{Y}_k^j$  represents  the waiting time between the appearance
			of the $k$th and the $(k+1)$th vertex in  $\{D_j(t)\}_{t\geq t_0^i}$. 			
			
			\begin{teo}
		        \label{S-Yule2}
				Let $z=\ln(1+y/(\ell^j_{k-1}-1))$, $k \ge 1$, $y > 0$, and $0<\varepsilon_t<1$ such that
				$t\varepsilon_t^2\longrightarrow\infty$ as $t\longrightarrow\infty$. Then,
				\begin{equation}
					\Big|\bbP(\mathcal{Y}_k^j\leq y)-\bbP(Z_k^j\leq z)\Big|<O\Big(\frac{1} {\ell^j_{k-1}\varepsilon_{\ell^j_{k-1}}^2}\Big),
				\end{equation}
				where $Z_k^j$ is an exponentially distributed random variable of parameter $k$.
			\end{teo}

 			\begin{obs}
 				\label{ObsS-Yule2}
				Since $t_0^i=\ell_0^j$ and $\ell_k^j>\ell_0^j$, $k\geq1$,	the previous theorem states  that for any $t^*=t_0^i$
				large enough but fixed,  
				\begin{equation}
					\Big|\bbP(\mathcal{Y}_k^j\leq y)-\bbP(Z_k^j\leq z)\Big|<O\Big(\frac{1}{t^*\varepsilon_{(t^*)^2}^2}\Big).
				\end{equation} 
				
				In words it means that from a fixed but large time $t^*$, all the waiting times $\mathcal{Y}_k^j$
				are approximately exponential random variables of parameter $k$, with an error term smaller than
				$O\left(1/(t^*\varepsilon_{(t^*)^2}^2)\right)$.  
				Thus,  for $t^*$ large enough we start to see a process  which is very  close to a Yule process with parameter $k$. 
			\end{obs}

			\begin{proof}[Proof of Theorem \ref{S-Yule2}]
				Let us define the Bernoulli random variables $\mathcal{X}^j_{k,\ell}$, $\ell\geq1$, with
				$\bbP(\mathcal{X}^j_{k,\ell}=1)=k\alpha/V(\ell^j_{k-1}+\ell)=1-\bbP(\mathcal{X}^j_{k,\ell} =0)$, so,
				$\{\mathcal{X}^j_{k,\ell}=1\}$ denotes the event that any of the $k$ descendant of $v_j$  in $\{D_j(t)\}_{t\geq t_0^i}$
				gives birth to a new one at time $\ell^j_{k-1}+\ell$.				
				Note that  the event $\{\mathcal{Y}_k^j=y\}$ is equivalent to  the event
				$\{\mathcal{X}^j_{k,1}=0,\mathcal{X}^j_{k,2}=0,\dots,\mathcal{X}^j_{k,y}=1\}$.
				Now define the events $\mathcal{E}_t:=\{t(\alpha-\varepsilon_t)\leq V(t)\leq t(\alpha+\varepsilon_t)\}$.
				By Chebyschev's inequality we have $\bbP(\mathcal{E}_t^c)\leq \alpha(1-\alpha)/t\varepsilon_t^2$,
				so $\bbP(\mathcal{E}_t)\longrightarrow 1$ if  $t\varepsilon_t^2\longrightarrow\infty$ as $t\longrightarrow\infty$.
				Then observe that 
				\begin{align*}
					\bbP(\mathcal{X}^j_{k,\ell}=x)\sim\bbP(\mathcal{X}^j_{k,\ell}=x\mid \mathcal{E}_{\ell^j_{k-1}+\ell-1})
					+O\Big(\frac{1}{(\ell^j_{k-1}+\ell-1)\varepsilon_{\ell^j_{k-1}+\ell-1}^2}\Big),
				\end{align*}
				and
				\begin{align}
					\label{YV}
					\bbP(\mathcal{Y}_k^j=y)\sim\Big[\bbP(\mathcal{X}^j_{k,y}=1\mid \mathcal{E}_{\ell^j_{k-1}+y-1})
					\prod_{\ell=1}^{y-1} \bbP(\mathcal{X}^j_{k,\ell}=0\mid \mathcal{E}_{\ell^j_{k-1}+\ell-1})\Big]
					+O\Big(\frac{1}{(\ell^j_{k-1})\varepsilon_{\ell^j_{k-1}}^2}\Big).
				\end{align}
				Assuming that $\varepsilon_{\ell^j_{k-1}+x-1}>\varepsilon_{\ell^j_{k-1}+x-2}>\dots>\varepsilon_{\ell^j_{k-1}}$,
				we obtain that the right side of  (\ref{YV}) is bounded above by
                \begin{align}
                	\label{Ymenor}
					& \frac{\alpha k}{(\ell^j_{k-1}+y-1)(\alpha-\varepsilon_{\ell^j_{k-1}+y-1})}
					\left(1-\frac{\alpha k}{(\ell^j_{k-1}+y-2)(\alpha+\varepsilon_{\ell^j_{k-1}+y-2})}\right) \nonumber\\
					& \qquad \times \dots \times \left(1-\frac{\alpha k}{\ell^j_{k-1}(\alpha+\varepsilon_{\ell^j_{k-1}})}\right)
					+O\Big(\frac{1}{\ell^j_{k-1}\varepsilon_{\ell^j_{k-1}}^2}\Big)\nonumber\\
					& = \frac{\alpha k}{(\ell^j_{k-1}+y-1)(\alpha-\varepsilon_{\ell^j_{k-1}+y-1})}
					\prod_{r=\ell^j_{k-1}}^{\ell^j_{k-1}+y-2}\left(1-\frac{\alpha k}{r(\alpha+\varepsilon_{r})}\right)
					+O\Big(\frac{1}{\ell^j_{k-1}\varepsilon_{\ell^j_{k-1}}^2}\Big),
				\end{align}
				and bounded below by
				\begin{align}
					\label{Ymayor}
					&\frac{\alpha k}{(\ell^j_{k-1}+y-1)(\alpha+\varepsilon_{\ell^j_{k-1}+y-1})}
					\left(1-\frac{\alpha k}{(\ell^j_{k-1}+y-2)(\alpha-\varepsilon_{\ell^j_{k-1}+y-2})}\right) \nonumber\\
					& \qquad \times \dots \times\left(1-\frac{\alpha k}{\ell^j_{k-1}(\alpha-\varepsilon_{\ell^j_{k-1}})}\right)
					+O\Big(\frac{1}{\ell^j_{k-1}\varepsilon_{\ell^j_{k-1}}^2}\Big)\nonumber\\
					& = \frac{\alpha k}{(\ell^j_{k-1}+y-1)(\alpha+\varepsilon_{\ell^j_{k-1}+y-1})}\prod_{r=\ell^j_{k-1}}^{\ell^j_{k-1}+y-2}
					\left(1-\frac{\alpha k}{r(\alpha-\varepsilon_{r})}\right)+O\Big(\frac{1}{\ell^j_{k-1}
					\varepsilon_{\ell^j_{k-1}}^2}\Big).
				\end{align}
				Thus, in a similar manner as we did in the proof of Theorem \ref{S-Yule}, by using Lemma \ref{lema1} and
				Euler--Maclaurin formula  to (\ref{Ymenor}) and (\ref{Ymayor}),   we find that
				\begin{align}
					\label{TTsim}
					\Big|\bbP(\mathcal{Y}_k^j\leq y)-[1-\exp(-k\ln(1+y/(\ell^j_{k-1}-1)))]\Big|<O\Big(\frac{1}{\ell^j_{k-1}
					\varepsilon_{\ell^j_{k-1}}^2}\Big).\nonumber
				\end{align}
				Then, taking $z=\ln(1+y/(\ell^j_{k-1}-1))$, we obtain
				\begin{align}
					\Big|\bbP(\mathcal{Y}_k^j\leq y)-\bbP(\mathcal{Z}_k^j\leq z)
					\Big|<O\Big(\frac{1}{\ell^j_{k-1}\varepsilon_{\ell^j_{k-1}}^2}\Big),\nonumber
				\end{align}  
				where $\mathcal{Z}_k^j$ is an exponentially distributed random variable with parameter $k$,	 which proves the thesis.
			\end{proof}

	\section{Discussion and conclusions}

		\label{finale}
		To compare the Barab\'asi--Albert and Simon models, we considered a third model that we
		called here the II-PA model, first introduced in \cite{Newman2005} with a different name. Then we gave a common description of
		the three models by introducing three different random graph processes related to them. This representation allowed us
		to clarify in which sense the three models can be related.  
		For each fixed time, if $m=1$, we proved that the Barab\'asi--Albert and the II-PA
		models have exactly the same preferential attachment probabilities (Theorem \ref{m=1}). Furthermore, since in the first model
		the preferential attachment is meant with respect to the whole degree of each vertex while in the second case it is meant with respect
		only to the in-degree, the conclusion is that, for a uniformly selected random vertex, the degree distribution
		in the Barab\'asi--Albert model equals the in-degree distribution in the Simon model.
		Note that $m=1$ is the only case in which this is true.

		Since the direct comparison between Barab\'asi--Albert and Simon model is not possible we first compared 
		II-PA model with Barab\'asi--Albert model (Theorem \ref{m=1}), and then II-PA model with Simon model (Theorem \ref{newman}).
		We underline that, even if the introduction of II-PA model was functional to the study of the connections between
		the Barab\'asi--Albert and Simon models, this hybrid model is interesting in itself. 
		
		Regarding the connections between Simon and II-PA models, Theorem \ref{newman} shows that when time goes to infinity,
		the II-PA model has the same limiting in-degree distribution as that of the Simon model with parameter
		$\alpha=1/(m+1)$, for any $m\geq1$. The proof uses the Azuma--Hoeffding concentration inequality
		and the supermartingale's convergence theorem.
				
		Combining  Theorem \ref{m=1} and  \ref{newman},  we conclude that, in the limit,
		the Simon model has the same in-degree distribution as that of the Barab\'asi--Albert model, for $\alpha=1/2$ and $m=1$.
		The existing relations between the three models are summarized in Figure \ref{angelica}.

		On the other hand, Yule model is defined in continuous time. In Section \ref{Relation2} we give
		a mathematical explanation of the reason why,
		when time goes to infinity the distribution of the size of a genus selected uniformly at random
		in the Yule model coincide with the in-degree distribution of Simon model.
		More precisely, we recognize which are the two different processes that describe Simon model and how they are related with a Yule model.
		Theorem \ref{S-Yule} and Theorem \ref{S-Yule2} show that,
		as time flows, these two different processes approximates the behavior of a continuous time process
		that in fact corresponds to a Yule model with parameters $(1-\alpha,1)$. This result is obtained in probability. 
			
		Many other preferential attachment models have appeared in the  literature in the last years. 
		In \cite{CooperFrieze} for instance,  a  general model of web graphs is studied. 
		With the right choice of the parameters this model includes the Barab\'asi--Albert model, however,  Simon and Yule models do not
		fit into the general set of assumptions considered in \cite{CooperFrieze}.  
		For a discussion of several related preferential attachment models  see for  example \cite{RemcoBook},  Chapter 8,  or
		\cite{Durrett2006}, Chapter 4.  

		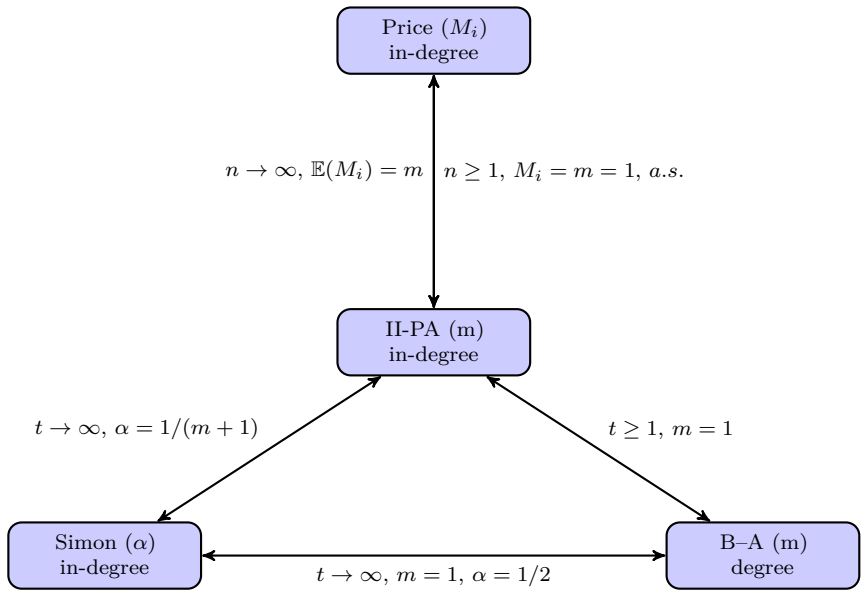
\begin{figure}[ht]
			\begin{center}
				\begin{tikzpicture}[<->,>=stealth',shorten >=0pt,auto,node distance=4.0cm,thick]
					\tikzstyle{every state}=[fill=blue!20!white,rectangle,rounded corners=5pt,font=\footnotesize, align=center]
					\node[state,text width=2.3cm, xshift=1.5cm] (A)   {Simon~($\alpha$)  in-degree};
					\node[state,text width=2.3cm, xshift=1.5cm] (B) [above right of=A] {II-PA~(m) in-degree};
					\node[state,text width=2.3cm, xshift=1.5cm] (C) [below right of=B] {B--A~(m) \\ degree};
					\node[state,text width=2.3cm, xshift=0cm] (D) [above  of=B] {Price~($M_i$) \\ in-degree};		
					\path (A)    edge node [above left,font=\footnotesize] {$t\rightarrow\infty$, $\alpha=1/(m+1)$}
						  (B)    edge node [below,font=\footnotesize] {$t\rightarrow\infty$, $m=1$, $\alpha=1/2$}  (C)
						  (B)    edge node [font=\footnotesize] {$t\geq1$, $m=1$}(C) 
						  (D)   edge node [above left, font=\footnotesize] {$n\rightarrow\infty$, $\bbE(M_i)=m$}(B)
						  (D)   edge node [above right, font=\footnotesize] {$n\geq1$, $M_i=m=1$, $a.s.$}(B);
				\end{tikzpicture}
			\end{center}
			\caption{\footnotesize{\label{angelica}The relations between Simon, II-PA, Price and Barab\'asi--Albert (B--A in the picture) models.
				Note that II-PA and Barab\'asi--Albert models can be put in relation for any time $t$ but just in the case $m=1$.
				Instead, the connections between II-PA and Simon models and Simon and Barab\'asi--Albert models, respectively,
				hold in the limit for $t$ going to infinity (w.r.t in-degree or degree distribution). We include also the Price model
				which is by definition equivalent to the II-PA if $M_i=m=1$ almost surely. Moreover, Price and II-PA models have the same limit
				in-degree distribution when $\bbE(M_i)=m$.}}
		\end{figure}


	\bibliographystyle{plain}
	\bibliography{pachon-polito-sacerdote-REVISED-LAST}

\end{document}